\providecommand{\ifshorten}{\iffalse}\providecommand{\ifonlythms}{\iffalse}
\ifshorten \documentclass[10pt]{amsart} \usepackage{a4wide} \else
\documentclass[11pt]{amsart}
\fi

\newcommand{\bp}{\begin{proof}} \newcommand{\ep}{\end{proof}}
\newcommand{\comment}[1]{}

\usepackage[usenames,dvipsnames]{pstricks}	
\usepackage[small]{caption}	
\usepackage{epsfig}

\usepackage{latexsym, amssymb, amsthm}	
\usepackage[centertags]{amsmath}
\usepackage{color}	
\usepackage[normalem]{ulem}
\usepackage{stmaryrd}	
\usepackage{bold-extra}	

\usepackage{hyperref} 
\usepackage{xcomment}
\usepackage{dsfont}	
\usepackage{enumerate}

\usepackage{abstract} 

\newcommand{\restricted}[1]{{\upharpoonright}_{#1}}
\newcommand{\measureskel}[2]{\mathrm{Sk}_{#2}(#1)}
\newcommand{\measureleaf}[2]{\mathrm{Lf}_{#2}(#1)}
\newcommand{\muskel}[1][T]{\measureskel{#1}{\mu}}
\newcommand{\muskelp}{\measureskel{T'}{\mu'}}
\newcommand{\muleaf}[1][T]{\measureleaf{#1}{\mu}}

\newcommand{\set}[2]{\{ #1 \mid #2 \}}
\newcommand{\bset}[2]{\bigl\{ #1 \bigm| #2 \bigr\}}
\newcommand{\Bset}[2]{\Bigl\{#1\Bigm|#2\Bigr\}}

\newcommand{\dis}{\mathrm{dis}}

\ifshorten \newcommand{\sm}{} \else \newcommand{\sm}{\smallskip} \fi

 \newif\ifpctex

\newcommand{\bnu}{\boldsymbol{\upsilon}} 
\newcommand{\bnux}{\bnu^\smallx}
\newcommand{\bnuT}{\bnu^{(T, \mu)}}
\newcommand{\bnuTinf}{\bnu^{(T_\infty, \mu_\infty)}}
\newcommand{\bnuas}[1][\smallx\!]{$\bnu^{#1}$\nbd almost}
\newcommand{\bnuTas}{$\bnuT$\nbd almost}
\newcommand{\bnuTinfas}{$\bnuTinf$\nbd almost}
\newcommand{\rspan}[1]{\llbracket #1 \rrbracket} 
\newcommand{\rspanu}{\rspan{\underline{u}}} 
\newcommand{\rspanU}{\rspan{\underline{U}}}
\newcommand{\rspanUN}[1][N]{\rspan{\underline{U}_{#1}}}
\newcommand{\taux}[1][\smallx]{\varsigma_{#1}}
\newcommand{\tauxN}[1][N]{\taux[\smallx_{#1}]}
\newcommand{\tauxn}[1][\smallx]{\tau_{#1}}
\newcommand{\tauxNn}[1][N]{\tauxn[\smallx_{#1}]}

\DeclareMathOperator{\supp}{supp}
\DeclareMathOperator*{\essup}{ess\,sup}
\DeclareMathOperator{\linhull}{span}
\DeclareMathOperator{\At}{At}

  \newtheorem{definition}{Definition}[section]
  \newtheorem{theorem}[definition]{Theorem}

  \newtheorem{cond}[definition]{Condition}
  \newtheorem{proposition}[definition]{Proposition}
  \newtheorem{lemma}[definition]{Lemma}
  \newtheorem{cor}[definition]{Corollary}
  
  \newcommand{\beCond}[2]{\Rand{\vspace{0,6cm}\tt #1}\begin{cond}[#2]
  \label{#1}} \theoremstyle{definition}
  \newtheorem{rem}[definition]{Remark}
  
  \newtheorem{ex}[definition]{Example}
  \newcommand{\exend}{\unskip\enskip\hbox{}\nobreak\hfill$\lozenge$}
  \newenvironment{example}{\begin{ex}}{\nopagebreak\exend\end{ex}}
  \newenvironment{remark}{\begin{rem}}{\exend\end{rem}}
  \numberwithin{equation}{section}
  \newtheoremstyle{step}{3pt}{0pt}{\itshape}{}{\bf}{}{.5em}{}

\theoremstyle{step} \newtheorem{step}{Step}

\newcommand{\E}{\mathbb{E}} \newcommand{\K}{\mathbb{K}}
\newcommand{\R}{\mathbb{R}} 
\newcommand{\N}{\mathbb{N}} \newcommand{\M}{\mathbb{M}}	

 \newcommand{\CE}{\mathcal{E}}

\newcommand{\CM}{\mathcal{M}}

\newcommand{\C}{\mathcal{C}}
\newcommand{\Cb}{\C_b}
\newcommand{\law}{\mathcal{L}}
\renewcommand{\P}{\mathcal{M}_1}
\newcommand{\F}{\mathcal{F}}
\newcommand{\A}{\mathcal{A}}
\newcommand{\G}{\mathcal{G}}
\newcommand{\Ft}{\tilde\F}
\newcommand{\Dom}{\mathcal{D}}

\newcommand{\smallcal}[1]{
	{\mathchoice{\scriptstyle}{\scriptstyle}{\scriptscriptstyle}{\scriptscriptstyle}\mathcal{#1}}}
\newcommand{\smallx}{\smallcal{X}}
\newcommand{\smally}{\smallcal{Y}}

\renewcommand{\H}{\mathbb{H}}		
\newcommand{\Sx}[1][\smallx]{\mathbb{S}_{#1}}
\newcommand{\Hbi}{\H^{f,\sigma}} 	
\newcommand{\Mn}[1][n]{\mathbb M_{#1}}
\newcommand{\Hn}[1][n]{\H_{#1}}
\newcommand{\Ht}[1][n]{\widetilde\H_{#1}}
\newcommand{\HK}[1][K]{\H^{#1, \sigma}}
\newcommand{\HC}[1][C]{\H^{K, \sigma}_{#1}}
\newcommand{\HCC}[1][C]{\H^{#1, \sigma}_{#1}}
\newcommand{\Pit}{\widetilde\Pi_n}
\newcommand{\Pin}{\Pi_n}
\newcommand{\Phit}{\tilde\Phi}
\newcommand{\Psit}{\tilde\Psi}
\newcommand{\Rtree}{$\protect\mathbb R$\nbd tree} 
\newcommand{\Rtrees}{\Rtree s} 
\newcommand{\LWVtopo}{{\rm LWV}\nbd topology}
\newcommand{\sGwtopo}{{\rm sGw}\nbd topology}
\newcommand{\pGwtopo}{{\rm pGw}\nbd topology}
\newcommand{\LWVconv}{{\rm LWV}\nbd convergence}
\newcommand{\sGwconv}{{\rm sGw}\nbd convergence}

\renewcommand{\(}{\bigl(}	\renewcommand{\)}{\bigr)}
\newcommand{\nbd}{\protect\nobreakdash-\hspace{0pt}}	

\newcommand{\Rand}[1]{\marginpar{#1}} 
\newcommand{\D}{\displaystyle}
\newcommand{\newatop}[2]{\underset{\scriptscriptstyle #2}{#1}}	
\newcommand{\convto}[2][\infty]{\,\newatop{\longrightarrow}{#2 \to #1}\,}	
\newcommand{\wconvto}[2][\infty]{\,\newatop{\Longrightarrow}{#2 \to #1}\,}	

\newcommand{\Tno}{\wconvto{n}}

\newcommand{\TNo}{\wconvto{N}}

\newcommand{\tNo}{\convto{N}}

\newcommand{\tol}{\stackrel\law\Rightarrow}
\newcommand{\tolN}[1][N]{\underset{#1\to\infty}{\stackrel\law\Longrightarrow}}
\newcommand{\tolNs}[1][N]{\;\tolN[#1]\;}
\newcommand{\toLWV}{\stackrel{\scriptscriptstyle\mathrm{LWV}\;}{\longrightarrow}}
\newcommand{\ToLWV}{\stackrel{\scriptscriptstyle\mathrm{LWV}\;}{\Longrightarrow}}
\newcommand{\toLWVN}[1][N]{\underset{\scriptscriptstyle #1\to\infty}{\toLWV}}
\newcommand{\ToLWVN}[1][N]{\underset{\scriptscriptstyle #1\to\infty}{\ToLWV}}
\newcommand{\topGwN}[1][N]{\underset{\scriptscriptstyle #1\to\infty}{\topGw}}

\newcommand{\toGwN}[1][N]{\underset{\scriptscriptstyle #1\to\infty}{\toGw}}

\newcommand{\topGw}{\stackrel{\scriptscriptstyle\mathrm{pGw}\;}{\longrightarrow}}
\newcommand{\tosGw}{\stackrel{\scriptscriptstyle\mathrm{sGw}\;}{\longrightarrow}}
\newcommand{\toGw}{\stackrel{\mathrm{Gw}\;}{\longrightarrow}}
\newcommand{\toSk}[1][]{\,\underset{\scriptscriptstyle\mathrm{#1}\;}{\stackrel{\mathrm{Sk}\;}{\Longrightarrow}}\,}

\newcommand{\toN}{\,\underset{N\to\infty}{\longrightarrow}\,}

\newcommand{\folge}[2][n]{(#2_{#1})_{#1\in\N}}
\newcommand{\dist}{\R_+^{\binom{\N_0}{2}}}

\newcommand{\distnm}{\R_+^{\binom{n+m+1}{2}}}

\newcommand{\nlim}{\lim_{n\to\infty}}

\newcommand{\mun}[1][]{\mu_{#1}^{\otimes n}}

\newcommand{\uvec}{\underline{u}}
\newcommand{\Uvec}{\underline{U}}
\newcommand{\vvec}{\underline{v}}
\newcommand{\vphi}{\varphi}
\newcommand{\eps}{\varepsilon}
\newcommand{\Psipol}[1][\gamma]{\Psi^{#1,n,\Phi}}
\newcommand{\Psitpol}[1][\gamma]{\Psi^{#1,n,\tilde\Phi}}
\newcommand{\Phipol}{\Phi^{\gamma,m,\vphi}}

\newcommand{\nuh}{\hat{\nu}}

\newcommand{\gammaPhi}{\gamma_{\Phit}}

\newcommand{\plainint}[2]{\int #1\;\mathrm{d}#2}
\newcommand{\plainintset}[3]{\int_{#1} #2\;\mathrm{d}#3}
\newcommand{\plainintmu}[2][]{\plainint{#2}{\mun[#1]}}
\newcommand{\plainintsetmu}[3][]{\plainintset{#2}{#3}{\mun[#1]}}

\newcommand{\integral}[3]{\int #1\;#2(\mathrm{d}#3)}
\newcommand{\intset}[4]{\int_{#1} #2\;#3(\mathrm{d}#4)}
\newcommand{\intmu}[2][]{\integral{#2}{\mun[#1]}{\uvec}}

\begin{document}
{\let\thefootnote\relax\footnotetext{Wolfgang L\"ohr was partially supported by the
European Science Foundation under the RGLIS short visit grant 5429}}
\title[The Pruning Process]
{{\Large Convergence of bi-measure $\mathbb R$-trees and {the} pruning process\footnote{Preprint of \emph{Ann.\
Inst.\ Henri Poincar\'e Probab.\ Stat. 51 (2015), no. 4, 1342--1368}}}}
\thispagestyle{empty}

\ifshorten\else
\author{Wolfgang L\"ohr}
\address{Wolfgang L\"ohr\\
Universit\"at Duisburg-Essen\\ Fakult\"at f\"ur Mathematik, Campus Essen\\ Universit\"atsstrasse 2\\ Essen 45141 Germany}
\email{wolfgang.loehr@uni-due.de}

\author{Guillaume Voisin}
\address{Guillaume Voisin\\
Universit\"at Duisburg-Essen\\ Fakult\"at f\"ur Mathematik, Campus Essen\\ Universit\"atsstrasse 2\\ Essen
45141\\
Germany - Universit\'{e} Paris Sud \\ Laboratoire de Math\'{e}matiques \\ 91405 \\ Orsay \\ France}
\email{guillaume.voisin1@u-psud.fr}

\author{Anita Winter}
\address{Anita Winter\\
Universit\"at Duisburg-Essen\\ Fakult\"at f\"ur Mathematik, Campus Essen\\ Universit\"atsstrasse 2\\ Essen 45141\\
Germany}
\email{anita.winter@uni-due.de}
\fi

\date{\today}

\ifshorten\else

\fi

\ifonlythms\maketitle\begin{xcomment}{lemma,theorem,proposition,definition,cor,caution}\fi

\maketitle

\begin{abstract} In \cite{AldousPitman1998} a tree-valued Markov chain is derived by pruning off more and more subtrees along the edges of a Galton-Watson tree.
More recently, in \cite{AbrahamDelmas2012}, a continuous analogue of the tree-valued pruning dynamics is constructed along L\'{e}vy  trees.
In the present paper, we provide a new topology which allows to link the discrete and the continuous dynamics by
considering them as instances of the same strong Markov process with different initial conditions.
We construct this pruning process on the space of so-called bi-measure trees, which are metric measure spaces
with an additional pruning measure. The pruning measure is assumed to be finite on finite trees, but not
necessarily locally finite. We also characterize the pruning process analytically via its Markovian generator
and show that it is continuous in the initial bi-measure tree.
A series of examples is given, which include the case where the pruning measure is the length measure on the
underlying tree.
\end{abstract}

\medskip
\renewcommand{\abstractname}{R\'{e}sum\'{e}}
\begin{abstract} Dans \cite{AldousPitman1998}, les auteurs obtiennent une cha\^{i}ne de Markov \`{a} valeurs arbres en \'{e}laguant de plus en plus de sous-arbres le long des n\oe uds d'un arbre de Galton-Watson. Plus r\'{e}cemment dans \cite{AbrahamDelmas2012}, un analogue continu de la dynamique d'\'{e}lagage \`{a} valeurs arbres est construit sur des arbres de L\'{e}vy. Dans cet article, nous pr\'{e}sentons une nouvelle topologie qui permet de relier les dynamiques discr\`{e}tes et continues en les consid\'{e}rant comme des exemples du m\^{e}me processus de Markov fort avec des conditions initiales diff\'{e}rentes. Nous construisons ce processus d'\'{e}lagage sur l'espace des arbres appel\'{e}s bi-mesur\'{e}s, qui sont des espaces m\'{e}triques mesur\'{e}s avec une mesure d'\'{e}lagage additionnelle. La mesure d'\'{e}lagage est suppos\'{e}e finie sur les arbres finis, mais pas n\'{e}cessairement localement finie. De plus, nous caract\'{e}risons analytiquement le processus d'\'{e}lagage par son g\'{e}n\'{e}rateur infinit\'{e}simal et montrons qu'il est continu en son arbre bi-mesur\'{e} initial. Plusieurs exemples sont donn\'{e}s, notamment le cas o\`{u} la mesure d'\'{e}lagage est la mesure des longueurs sur l'arbre sous-jacent.
\end{abstract}

\newpage

{\tiny
\begin{tableofcontents}
\end{tableofcontents}
}

\section{Introduction and motivation}
\label{S:intro}

Let ${\mathcal G}_1$ be a rooted Galton-Watson tree with an offspring generating function $g$. For $0\le u\le 1$, let ${\mathcal G}_u$ be the subtree of ${\mathcal G}_1$ obtained by retaining each edge with probability $u$. Lyons (\cite{Lyons1992}) showed that ${\mathcal G}_u$ is again a Galton-Watson tree which corresponds to the offspring generating function $g_u=g(1-u+u\cdot)$.
As one can couple the pruning procedures for several $u\in[0,1]$ in such a way that ${\mathcal G}_{u'}$ is a rooted subtree of ${\mathcal G}_{u}$ whenever $u'\le u$, they give rise to a non-decreasing tree-valued Markov process $({\mathcal G}_u)_{u\in[0,1]}$ which was further studied in Aldous and
Pitman (\cite{AldousPitman1998}). Recently, Abraham, Delmas and He consider in \cite{AbrDelHe2012} another
pruning procedure on Galton Watson trees
where cut points fall on the branch points to the effect that the  subtree above is pruned.
Here each node of the initial Galton-Watson tree is cut
independently with probability
$1-u^{n-1}$ where $n$ is the number of children of the node.

In the same spirit some authors
\comment{Abraham and Delmas constructed}
consider continuum tree analogues of pruning dynamics. Compare, for example,
\cite{AldousPitman1998b,AbrahamSerlet2002} for a pruning proportional to the length on the skeleton of a
Brownian CRT, \cite{Miermont2005} for a pruning on the infinite branch points of a stable L\'{e}vy tree,
\cite{AbrDel2008} for a pruning on the infinite branch points of a L\'{e}vy tree without Brownian part,
\cite{AbrDelVoi2010, AbrahamDelmas2012} for a combined pruning proportional to the length and on the infinite branch
points of a general L\'{e}vy tree.

In \cite{AbrahamDelmas2012} it is conjectured that the pruning procedure presented in the same paper is the
continuous analogue of a mixture of the pruning procedures
suggested in \cite{AldousPitman1998} and \cite{AbrahamDelmas2012}, that is of pruning procedures on Galton-Watson trees where cut points fall on edges as well as on nodes.
However, no precise link between the discrete and the continuum tree-valued dynamics has been given so far. The main goal of the  present paper is to present
{\sc one} Markov process, which in the following is referred to as {\sc the} \emph{pruning process}.
We shall give an analytic characterization via a Markovian generator and
provide with the so-called \emph{leaf-sampling weak vague topology}
a notion of convergence which allows to state convergence of the discrete tree-valued dynamics to the associated continuous tree-valued dynamics.

It had been a long tradition to encode trees via continuous excursions, and to use uniform topology as a notion of convergence. A more recent and conceptional approach is to
think of trees as ``tree-like'' metric spaces, the so-called $\R$-trees, and to use the Gromov-Hausdorff topology as a notion of convergence
(compare, for example, \cite{DreTer96} for an introduction into $\R$-trees and \cite{Gromov2000,EvaPitWin2006} for details on the Gromov-Hausdorff distance).
For a long time convergence of suitably rescaled Galton-Watson processes were established for very particular offspring distributions only. To be in a position
to prove an invariance principle, Aldous developed in \cite{Ald1991a,Ald1993} a notion of convergence by encoding  trees as closed subsets of $l_+^1$, the space of summable sequences of positive numbers which were additionally equipped with a sampling measure. Convergence was then proposed as the convergence of all subtrees  spanned by finite samples from the tree. Once more, this very neat and powerful idea had been generalized to the more conceptional encoding of trees as metric probability measure spaces where the tree space was equipped with the so-called \emph{Gromov-weak topology} (compare \cite{Gromov2000,GrePfaWin2009}).
Further developments which combine the Gromov-Hausdorff and Gromov-weak topology and allow for sampling measures
that are finite on bounded sets can be found, for example, in \cite{EvaWin2006,Miermont2009,AbrDelHos2013}.

In the present paper, we provide a unified framework by regarding these pruning processes as the \emph{same}
Feller-continuous Markov process on a (non locally compact) space of \Rtree s with different initial conditions,
and to establish convergence in Skorohod space whenever the initial distribution converges.
For that purpose, we introduce \emph{bi-measure $\R$-trees}, i.e., metric measure spaces $(T,r,\mu)$, which are
additionally equipped with a so-called \emph{pruning} measure, $\nu$. Here, the so-called \emph{sampling measure}
$\mu$ is a finite measure (allowing for a varying total mass), while the pruning measure is only assumed to be
finite on finite subtrees.
As the pruning measure is already part of the state, we are in a position to construct one (universal) pruning process. This process is a pure jump process which, given a bi-measure $\R$-tree, lets rain down
successively more and more cut points according to a Poisson process whose intensity measure is equal to the
pruning measure. At each cut point, the subtree above is cut off and removed, and the sampling and pruning
measures are simultaneously updated by simply restricting them to the remaining, pruned part of the tree.

A major difficulty is that important examples for the pruning measures, such as the length measure on the Brownian CRT, are not locally finite.
Therefore, we introduce with the \emph{leaf-sampling weak vague topology} a new topology on the spaces of bi-measure \Rtree s.
We give equivalent characterizations of convergence and provide convergence determining classes of functions.

\bigskip

\noindent{\bf Outline. }
The paper is organized as follows. In Section~\ref{S:topology} we introduce the leaf-sampling weak-vague topology and give a characterization of convergence.
In Section~\ref{S:pruning} we construct the pruning process, calculate its Markovian generator and verify that the law of the process on Skorohod space
depends continuously on the initial condition. Finally, in Section~\ref{sec:examples} we apply our main result
to obtain convergence of various pruning processes that appeared in the literature.

\comment{The space of bi-measure \Rtree s we have
constructed allows us to construct all the procedures using the $\sigma$-finite measure as the pruning measure.
Our main theorem says that the pruning process is a Feller process, i.e.\ its distribution depends continuously on its initial condition. In other terms, if a sequence of bi-measure \Rtree s converges in the {\rm LWV}-topology, then the sequence of associated pruning processes converges in the Skorohod space. As an example, we know that a properly normalized sequence of Galton-Watson endowed with the length measure converges to the Brownian CRT, then the discrete prunings of Aldous and Pitman and of Abraham, Delmas and He converge to the pruning process of Abraham and Serlet.}

\comment{
 Call $\mathcal G_u$ the connected subtree of $\mathcal G$ containing the root after removing the edges of $\mathcal G$ independently with probability $1-u$. They obtain a non-decreasing tree-valued Markov process $(\mathcal G_u^{AP},0\leq u \leq 1)$ in the sense that $\mathcal G_v^{AP}$ is a subtree of $\mathcal G_w^{AP}$ if $v<w$. They deal with the Poisson law because it's the only distribution such that the trees $\mathcal G_u^{AP}$ have the same law for every $0\leq u\leq 1$.

More recently, Abraham, Delmas and He \cite{AbrDelHe2012} define a non-decreasing tree-valued Markov process
$(\mathcal G_u^{ADH},0\leq u \leq 1)$ using the same kind of pruning but on the vertices of a Galton-Watson tree
with any reproduction law. They prove that for $0\leq v<w\leq 1$, conditionally on the number of leaves, the
trees $\mathcal G_v^{ADH}$ and $\mathcal G_w^{ADH}$ have the same law if and only if there exists a time $u$
such that the nodes of $\mathcal G_w^{ADH}$ with $n$ children are cut with probability $1-u^{n-1}$. This implies
that this is the good pruning on the nodes to obtain a Markovian process.

Galton-Watson trees, properly renormalized, converge (in a weak sense defined later) to a continuum random tree. Then continuous pruning procedures have been defined on continuous trees. Abraham and Delmas \cite{AbrahamDelmas2012} construct a decreasing L\'{e}vy-CRT-valued Markov process $(\mathcal T_\theta,\theta\geq 0)$ using a pruning procedure on the nodes and on the height of a L\'{e}vy  tree. They prune the infinite nodes independently using an exponential clock i.e.\ each node $s$ is cut with probability $1-e^{-\theta \Delta_s}$ where $\Delta_s$ is the ``mass'' of the node $s$. The lineage of each individual $s$ of the tree is cut at the atoms of a Poisson process with parameter proportional to $\theta H_s$ where $H_s$ is the distance from $s$ to the root given by the height process. They use a L\'{e}vy  snake to have a pruning procedure adapted to the genealogical property of the tree.

Although the process $(\mathcal T_\theta,\theta\geq0)$ is the continuous analogue of $\mathcal G^{ADH}$ and $\mathcal G^{AP}$ (or maybe both of them), no link between them is actually proved. In this paper, we focus on the discrete pruning $\mathcal G^{ADH}$ on nodes and its convergence toward a continuous one.
}

\section{Bi-measure $\mathbb R$-trees and the LWV-topology}
\label{S:topology}
In this section we introduce the space of $\mathbb{R}$-trees equipped with a finite \emph{sampling measure} and
a \emph{pruning measure} which is assumed to be finite on finite subtrees. Moreover, we define the {\em leaf-sampling
weak vague topology} (LWV-topology) on this space of bi-measure $\R$-trees.
The idea behind our topology is to first sample a finite number of points from the tree according to the
sampling measure. These points span a finite subtree. In many relevant examples they are actually the leaves of this subtree.
Then we equip this finite subtree with the restriction of the pruning
measure and obtain a random metric measure tree.
For convergence of bi-measure trees, we require that these random metric
measure trees converge together with the sampled points as \emph{$n$-pointed metric measure \Rtree s} in the
Gromov-weak topology.

We therefore recall in Subsection~\ref{Sub:Gw} the notion of Gromov-weak topology on metric measure spaces and extend it to the \emph{$n$-pointed Gromov-weak topology}.
In Subsection~\ref{sec:measureRtree} we then define a stronger topology on $n$-pointed metric measure \Rtree s,
the subtree Gromov-weak topology. Finally, in Subsection~\ref{subsec:LWV-topo} we define the \LWVconv.
It turns out that it can be characterized by both the pointed as well as the subtree
Gromov-weak convergence of samples from the bi-measure $\R$-tree and defines a separable, metrizable topology.
In Subsection~\ref{sec:convdet}, we introduce classes of test functions that induce the \LWVtopo. One of
them turns out to be convergence determining. Using these test functions, we derive several convergence results.

\subsection{The $n$-pointed Gromov-weak topology}
\label{Sub:Gw}
Greven, Pfaffelhuber and Winter~\cite{GrePfaWin2009} define the space of metric probability measure spaces
equipped with the Gromov-weak topology. In this subsection, we define a slightly more general space using finite
measures instead of probability measures and considering $n$-pointed metric measure spaces. We do not give
proofs, because the extension is straightforward.

We start recalling basic notation. As usual, given a topological space $X$, we denote by $\C(X)$ ($\Cb(X)$) the space of (bounded)
continuous, $\R$-valued functions on $X$, and by ${\mathcal M}_1(X)$ ($\mathcal M_f(X)$) the
space of probability (finite) measures, defined on the Borel $\sigma$\nbd algebra of $X$. For $x\in X$,
$\delta_x\in\CM_1(X)$ is the Dirac measure in the point $x$. ``$\Rightarrow$'' means weak convergence in $\mathcal M_1(X)$
or in $\mathcal M_f(X)$. Recall that the support of $\mu$, $\supp(\mu)$,
is the smallest closed set $X_0\subseteq X$ such that $\mu(X_0)=\mu(X)=:\|\mu\|$. For $\mu\in \mathcal M_f(X)$,
we denote the normalization by
\begin{equation}
\label{e:sampling}
   \mu^\circ
 :=
    \tfrac{\mu }{ \|\mu\|}\in \mathcal M_1(X).
\end{equation}

The {\em push forward} of $\mu$ under a measurable map $\phi$ from $X$ into another topological space
$Z$ is the finite measure $\phi_\ast\mu\in{\mathcal M}_f(Z)$ defined by
\begin{equation}
\label{push}
  \phi_\ast\mu(A)
 :=
   \mu\bigl(\phi^{-1}(A)\bigr),
\end{equation}
for all measurable subsets $A\subseteq Z$. For the integral of an integrable function $\vphi$ with respect to
$\mu$, we sometimes use the notation
\begin{equation}
	\langle\mu, \vphi\rangle := \plainint\vphi\mu.
\end{equation}

A \emph{metric measure space} is a triple $(X,r,\mu)$, where $(X,r)$ is a metric space such that
$\(\supp(\mu),r\)$ is complete and separable and $\mu\in{\mathcal M}_f(X)$ is a finite measure on $(X,{\mathcal
B}(X))$. If $\supp(\mu)$ is separable but not complete, we simply identify it with its completion.

Branching trees such as Galton-Watson trees and the CRT are often rooted. We therefore define a \emph{rooted
metric measure space} $(X,r,\rho,\mu)$ as a metric measure space $(X, r, \mu)$ together with a distinguished
point $\rho\in X$ which is referred to as the \emph{root}.
To avoid heavy notations, in the following we suppress the metric and the root, i.e.\ we abbreviate, for example,
\begin{equation}
\label{e:leightnotation}
  X=(X,r,\rho),	\qquad (X,\mu)=(X,r,\rho,\mu).
\end{equation}
The definition of metric measure spaces given in \cite{GrePfaWin2009} can easily be extended to rooted metric
measure spaces.
In the context of metric spaces, rooted spaces are often referred to as \emph{pointed spaces} (compare, for example, Section~8 in \cite{BurBurIva01}).

We want to extend these rooted metric measure spaces $(X,r,\mu)$ by fixing $n$ additional points
$u_1,\ldots,u_n\in X$, and call $(X,r,\rho,(u_1,\ldots,u_n),\mu)$ a (rooted)
\emph{$n$-pointed metric measure space}. The support of an $n$-pointed metric measure space
$(X,r,\rho,(u_1,\ldots,u_n),\mu)$ is defined by
\begin{equation}
\label{e:npointsupp}
   \supp\bigl((X,r,\rho,(u_1,\ldots,u_n),\mu)\bigr) := \supp(\mu)\cup\{\rho,u_1,\ldots,u_n\}.
\end{equation}

In the following we identify two $n$-pointed metric measure spaces if there is a measure preserving isometry between their
supports that also preserves the root and the fixed points.

\begin{definition}[The space $\Mn$]
Two $n$-pointed metric measure spaces $\smallx=(X,r,\rho, (u_1,\ldots ,u_n),\mu)$ and\/
$\smallx'=(X',r',\rho', (u_1',\ldots ,u_n'),\mu')$ are called \emph{equivalent} if there exists an isometry $\phi$ between
$\supp(\smallx)$ and\/ $\supp(\smallx')$ such that\/ $\phi_\ast\mu=\mu'$, $\phi(\rho)=\rho'$ and\/ $\phi(u_k)=u_k'$
for all\/ $1\leq k\leq n$. It is clear that this defines an equivalence relation.

We denote by\/ $\Mn$ the set of equivalence classes of\/ $n$-pointed metric measure spaces.
\label{Def:001}
\end{definition}\sm

\begin{remark} Notice that for  a notion of  equivalence of metric measure spaces $(X,r,\mu)$ and $(X',r',\mu')$ there are two canonical choices.
Either we insist that the metric spaces $(X,r)$ and $(X',r')$ are isometric or we are satisfied with their supports
to be isometric thereby neglecting sets of measure zero (compare, for example, \cite[Section~27]{Villani2009}).
Here we take the second approach which  allows for a characterization of convergence in $\mathbb{M}_n$
through convergence determining classes of functions. The gap between such a notion of (weak) convergence and a
stronger topology which also requires the convergence of supports of the measures is closed in \cite{AthreyaLoehrWinter}. 
\end{remark}\sm

To simplify notations, we do not distinguish between an $n$-pointed metric measure space and its equivalence
class. That is, we write
\begin{equation}
\label{e:leightsmallx}
  \smallx=(X,(x_1,\ldots,x_n),\mu)\in \Mn.
\end{equation}

\begin{remark}[The space $\M_0$]
	$\Mn[0]$ is the usual space of rooted metric measure spaces (with finite measures).
\label{Rem:001}
\end{remark}\sm

For a rooted metric space $X$, we define a map $R^X$ that associates to a sequence of points the matrix of their
distances to the root and to each other, i.e.,
\begin{equation}\label{mono1a}
     R^X\colon
     \left\{\begin{array}{ccl}X^{\mathbb N} &\to& \mathbb R_+^{\binom{\mathbb N_0}{2}} \\
     (x_i)_{i\geq 1} &\mapsto& \bigl(r(x_i, x_j)\bigr)_{0\leq i<j} \text{ with }x_0:=\rho.  \end{array}\right.
\end{equation}
The \emph{distance matrix distribution} of an $n$-pointed metric measure space $\smallx=(X,(u_1,\ldots,u_n),\mu)$
is then given by
\begin{equation}
\label{fct:rootedR}
   \bnu^{\smallx}
 :=
   \|\mu\|\cdot \bigl(R^X\bigr)_\ast\Bigl(\bigotimes_{k=1}^n \delta_{u_k} \otimes (\mu^\circ)^{\otimes\N}
   \Bigr)\in{\mathcal M}_f(\mathbb R_+^{\binom{\N_0}{2}}),
\end{equation}
which obviously depends only on the equivalence class. Vershik's proof of Gromov's reconstruction theorem
(see \cite[$3\frac12.7$]{Gromov2000}) directly carries over to $n$-pointed metric measure spaces.
Therefore, $\smallx\in\Mn$ is uniquely determined by its distance matrix distribution $\bnu^\smallx$.

\begin{definition}[\pGwtopo]
A sequence of $n$-pointed metric measure spaces $\smallx_N\in \Mn$ converges $n$-pointed Gromov-weakly (\rm pGw)
to $\smallx\in\Mn$ if
\begin{equation}
\label{e:pGw}
	\bnu^{\smallx_N}\underset{N\rightarrow \infty}{\Longrightarrow} \bnu^{\smallx}
\end{equation}
in the weak topology on $\mathcal M_f\(\mathbb R_+^{\binom{\mathbb N_0}{2}}\)$.
\label{Def:002}
\end{definition}\sm

We see directly from the definition that functions of the form $\Phi\colon \Mn\to\R,\, \smallx\mapsto
\langle\bnux,f\rangle$ with $f\in\Cb\(\dist\)$ are continuous. If $f$ depends only on finitely many coordinates,
$\Phi$ is called a \emph{polynomial}, and there exists $m\in\N$, $\vphi\in\Cb\(\distnm\)$ such that for
$\smallx=(X, \uvec, \mu)\in\Mn$,
\begin{equation}
\label{e:polynomial}
	\Phi(\smallx)=\Phi^{m,\varphi}(\smallx)
 :=
    \int_{X^m}\mu^{\otimes m}(\mathrm{d}\vvec)\,\varphi\bigl( R^X(\uvec,\vvec)\bigr),
\end{equation}
where $\uvec=(u_1,...,u_n)$, $\vvec=(v_1,...,v_m)$ and $(\uvec,\vvec):=(u_1,\ldots,u_n,v_1,\ldots,v_m)$.
Note that
\begin{equation} \label{e:phi1}
   \Phi^{m,1}(\smallx)=\|\mu\|^m,
\end{equation}
and, in particular, $\Phi^{0,1}\equiv 1$.
Moreover, as polynomials are not bounded (compare~(\ref{e:phi1})), we define a class $\Pin\subseteq \Cb(\Mn)$ of
bounded test functions by
\begin{equation} \label{e:Pi}
\begin{aligned}
	\Pin:=\big\{&\Phipol(\smallx):=\gamma(\|\mu\|)\cdot\Phi^{m,\vphi}\bigl(\smallx\bigr): \\
		&\Phi^{m,\vphi}\text{ is a polynomial},\, \gamma\in \Cb(\R_+),\,
	    \lim_{x\rightarrow \infty} x^k \gamma(x)=0,\,\forall k\in\N\big\}.
\end{aligned}
\end{equation}

Recall the {\em Prohorov distance} $d_{\rm Pr}$ between two finite measures $\mu,\nu$ on a metric space $(Z,d)$,
\begin{equation}
\label{e:dPr}
\begin{aligned}
   &d^{(Z,d)}_{\rm Pr}(\mu,\nu)
   \\
 &:=
   \inf\bigl\{\varepsilon>0:\;\mu(A^\varepsilon)+\varepsilon\geq\nu(A),\;\nu(A^\varepsilon)+\varepsilon \geq
   \mu(A)\, \forall A \text{ closed}\bigr\},
\end{aligned}
\end{equation}
where $A^\varepsilon:=\{x\in Z \mid d(x,A) < \varepsilon\}$.

\begin{definition}[$n$-pointed Gromov-Prohorov distance]
We define the $n$-pointed Gromov-Prohorov distance between $\smallx=(X,\uvec,\mu)$ and $\smally=(Y,\vvec,\nu)$ in $\Mn$ by
\begin{equation}
\label{e:GPdistance}
	d_{\rm pGP}\bigl(\smallx,\smally\bigr)
 :=
    \inf_{d}\bigl\{ d^{(X\uplus Y,d)}_{\rm Pr}(\mu,\nu)+d(\rho_X,\rho_Y) + \sum_{k=1}^n d(u_k,v_k)\bigr\},
\end{equation}
where the infimum is taken over all metrics $d$ on the disjoint union $X\uplus Y$ that extends $r_X$ and $r_Y$. If there is no confusion, we simply write $d_{\rm Pr}$ for $d^{(X\uplus Y,d)}_{\rm Pr}$.
\label{def:pGP}
\end{definition}\sm

Recall that a set $\F\subseteq \Cb(X)$ is \emph{convergence determining} (on the topological space $X$) if, for
probability measures $\mu_N, \mu$ on $X$, the weak convergence $\mu_N\TNo \mu$ is equivalent to
$\plainint{f}{\mu_N}\toN\plainint{f}{\mu}$ for all $f\in\F$.

\begin{proposition}[$\Pin$ is convergence determining]
Let $\smallx,\smallx_1,\smallx_2,\ldots\in \Mn$. The following conditions are equivalent:
\begin{enumerate}
	\item $\smallx_N \topGw \smallx$, as $N\to\infty$.
	\item $\Phi(\smallx_N) \tNo \Phi(\smallx)$, for all polynomials $\Phi$.
	\item $d_{\rm pGP}(\smallx_N,\smallx) \tNo 0$.
\end{enumerate}
Furthermore, $\Mn$ is separable, $d_{\rm pGP}$ is a complete metric on $\Mn$, and the class\/ $\Pin\subseteq
\Cb(\Mn)$ is convergence determining on $\Mn$.
\label{prop:pGwcharacterizations}
\end{proposition}

\begin{proof}
The proof of the equivalences is an obvious modification of that of Theorem 5 in \cite{GrePfaWin2009}. Notice
that $\mu^\circ$ in the definition of the \pGwtopo\ can be replaced by $\mu$ in the definition of polynomials because
$\|\mu\|=\Phi^{1,1}(\smallx)$. Separability and completeness follow in the same way as Proposition~5.6 in \cite{GrePfaWin2009}.

To see that $\Pin$ induces the \pGwtopo, note that $\Phi^{\gamma,0,1}\in\Pin$, and convergence of
$\Phi^{\gamma,0,1}(\smallx_N)=\gamma\(\|\mu_N\|\)$ with $\gamma(x)=e^{-x}$ implies the convergence of $\|\mu_N\|$.
Hence, the topology induced by $\Pin$ coincides with the topology induced by the polynomials. Using the fact
that $\Pin$ is multiplicatively closed, we see that it is convergence determining with the same proof as
for the set of polynomials on the space of metric probability measure spaces (see
\cite{DGP:mmmspace,Loehr:Gromovmetric}), or directly from Le~Cam's theorem (see \cite{LeCam57}, \cite[Lem.~4.1]{HoffmannJorgensen1977}).
\end{proof}\sm

\subsection{Measure $\R$-trees and subtree Gromov-weak topology}
\label{sec:measureRtree}
In this subsection we define the \emph{subtree Gromov-weak topology}. As ``tree-like'' metric spaces are $0$-hyperbolic,
throughout the paper we work with the subspaces
\begin{equation}
\label{w:002}
\begin{aligned}
  \Hn := \bigl\{\smallx\in\Mn:\,\smallx\text{ is $0$-hyperbolic}\bigr\}\subseteq\Mn,
\end{aligned}
\end{equation}
and
\begin{equation}
\label{w:002a}
\begin{aligned}
  \H:=\Hn[0]\subseteq\Mn[0],
\end{aligned}
\end{equation}
where a metric measure space $\smallx\in \Mn$ is called \emph{$0$\nbd hyperbolic} iff
\begin{equation}
\label{grev30}
    r(x_1,x_2)+r(x_3,x_4) \leq  \max \{r(x_1,x_3)+r(x_2,x_4),r(x_1,x_4)+ r(x_2,x_3) \},
\end{equation}
for all $x_1, x_2, x_3, x_4 \in \supp(\smallx)$.
It follows immediately from Theorem~2.5 in \cite{EvaPitWin2006} that for each $n\in\mathbb{N}$, $(\Hn,d_{\rm pGP})$ is complete.

Recall that a $0$\nbd hyperbolic space is called an \emph{\Rtree} if it is connected (see \cite{DreMouTer96}
for equivalent definitions and background on \Rtrees).
Given a (rooted) \Rtree\ $(T,r,\rho)$, we denote the unique path between two points $x,y\in T$ by $[x,y]$, and
${[x,y[}:=[x,y]\setminus\{y\}$.  The set of {\em leaves} of the tree is
\begin{equation}
\label{e:leaves}
   {\rm Lf}(T)
 :=
   T\setminus \bigcup_{x\in T} {[\rho,x[}.
\end{equation}
We also use the notation $\rspan{\vvec}$ for the tree spanned by the root $\rho$ and the vector $\vvec\in T^n$, i.e.,
\begin{equation}
\label{e:treespanned}
   \rspan{\vvec}:=\bigcup_{i=1}^n[ \rho,v_i ].
\end{equation}

Here and in the following we refer to any $\R$-tree of the form (\ref{e:treespanned}) as a {\em finite tree}.

\begin{remark}[$0$-hyperbolic spaces are equivalent to \Rtrees]
	According to Theorem~3.38 of \cite{EV2007}, every $0$\nbd hyperbolic space can be isometrically embedded
	into an \Rtree. Since our notion of equivalence of two $n$-pointed metric measure spaces $\smallx$
	and $\smallx'$ requires only a (measure and point preserving) isometry between $\supp(\smallx)$ and
	$\supp(\smallx')$, this means that every $n$-pointed $0$-hyperbolic metric measure space is equivalent
	to an $n$-pointed, measured \Rtree.
	In the following we assume, without loss of generality, that $\smallx\in \Hn$ is an $\R$-tree, by
	choosing a connected representative of the equivalence class.

	Also note that, given two \Rtrees\ $(T,r)$, $(T', r')$ with subsets $A\subseteq T$, $A'\subseteq T'$,
	and an isometry $\phi\colon A \to A'$, there is a unique extension of $\phi$ to an isometry between the generated \Rtrees,
	$\overline{\phi}\colon \rspan{A} \to \rspan{A'}$. Indeed, for $v\in \rspan{A}$ there exist (non-unique)
	$x,y\in A$ with $v\in[x,y]$, and a unique $w_v\in [\phi(x),\phi(y)]$ with $r(x,v)=r'\(\phi(x),
	w_v\)$. It is straightforward to check that $w_v$ does not depend on the choice of $x,y$ and
	$\overline{\phi}(v) := w_v$ is an isometry.
	In particular, for $\smallx\in\Hn$, the \Rtree\ $\rspan{\supp(\smallx)}$ is unique up to isometry.
\label{rem:000}
\end{remark}

We now define a topology on $\Hn$ which requires that every subtree generated by a subset of the $n$
distinguished points converges.
For that purpose, we define a projection map which sends a list $\uvec$ to the
sublist indexed by $I=\{i_1,...,i_k\}$
for given $1\leq i_1<\dots <i_k\leq n$. That is,
\begin{equation}
\label{e:piIn}
   \pi_I^n\colon \left\{\begin{matrix}T^n &\to& T^k\\ \uvec&\mapsto&(u_{i_1},\ldots,u_{i_k})\end{matrix}\right..
\end{equation}
The sublist $(u_1,\dots,u_k)$ of $\uvec\in T^n$ is simply denoted $\uvec^k$. With a slight abuse of notation, we also write
\begin{equation}
\label{e:piInsmallx}
   \pi_I^n(T,\uvec,\mu)
   := \bigl(\rspan{\pi^n_I(\uvec)},\pi^n_I(\uvec),\mu\bigr)
   := \bigl(\rspan{\pi^n_I(\uvec)},\pi^n_I(\uvec),\mu\restricted{\rspan{\pi^n_I(\uvec)}}\bigr),
\end{equation}
where the measure $\mu$ in the middle expression is tacitly understood to be restricted to the appropriate
space, $\rspan{\pi^n_I(\uvec)}$.

\begin{definition}[\sGwtopo]
Consider $n$-pointed measure\/ \Rtree s $\smallx$, $\smallx_1,\smallx_2,\ldots\in\Hn$.
We say that\/ $(\smallx_N)_{N\in\mathbb{N}}$ converges subtree Gromov-weakly ({\rm sGw}) to
$\smallx$ iff\/ $\smallx_N \topGwN \smallx$ and
\label{def:sGw}
\begin{equation}
\label{e:sGw}
	\pi_I^n(\smallx_N) \topGwN \pi_I^n(\smallx),\; \quad \forall I\subseteq \{1,...,n\}.
\end{equation}
\end{definition}\sm

Put
\begin{equation}
\label{e:032}
   \Ht:=\big\{(T, \uvec, \mu)\in \Hn:\,\supp(\mu)\subseteq \rspanu\big\}\subseteq \Hn,
\end{equation}
and note that $\Ht$ consists only of finite trees with at most $n$ leaves.

\begin{remark}[Related topologies]
	The \sGwtopo\ is strictly stronger than the \pGwtopo. On $\Ht$, \sGwconv\ implies measured
	Gromov-Hausdorff convergence (\cite{Fukaya1987}), also known as weighted Gromov-Hausdorff convergence
	(\cite{EvaWin2006,Miermont2009}).
\label{Rem:002}
\end{remark}\sm

\begin{lemma}[Sufficient condition for \sGwconv]
Consider random\/ $n$-pointed measure \Rtree s $\smallx=(T,\uvec,\mu),\, \smallx_N=(T_N,\uvec_N,\mu_N)\in \Ht$,
$N\in\N$ (in particular $T_N=\rspan{\uvec_N}$).
Assume that $(\smallx_N)_{N\in\mathbb{N}}$ converges almost surely (a.s.) to $\smallx$ in the $n$-pointed
Gromov-weak topology, as $N\to\infty$.
Furthermore, assume that there is a strictly increasing function $\psi\colon\R_+ \to \R_+$ such
that $\psi\(\|\mu\|\)$ is integrable and
\begin{equation} \label{e:cond(i)}
   \E\Bigl[\psi\Bigl(\mu_N\(\rspan{\pi_I^n(\uvec_N)}\)\Bigr)\Bigr] \tNo
   	\E\Bigl[\psi\Bigl(\mu\(\rspan{\pi_I^n(\uvec)}\) \Bigr)\Bigr], \;\; \forall I\subseteq\{1,...,n\}.
\end{equation}
\label{lem:lemma42}
Then $(\smallx_N)_{N\in\mathbb{N}}$ converges also subtree Gromov-weakly to $\smallx$, a.s., as $N\to\infty$.
\end{lemma}\sm

To prepare the proof, we state the following:
\begin{remark}[pGw-convergence yields a tree homomorphism]
Consider a sequence of $n$-pointed measure \Rtree s $\smallx=(T,\uvec,\mu)$, $\smallx_1=(T_1,\uvec_1,\mu_1)$,
$\smallx_2=(T_2,\uvec_2,\mu_2),\ldots \in \Ht$. Assume furthermore that $(\smallx_N)_{N\in\mathbb{N}}$ converges $n$-pointed
Gromov-weakly to $\smallx$, a.s., as $N\to\infty$.

For sufficiently large $N\in\mathbb{N}$, we can  define a function $f_N\colon T_N\to T$ by sending the root to the root, letting $f_N(u_{N,k})=u_k$ and $f_N(u_{N,k}\wedge u_{N,l})=u_k\wedge u_l$, $k,l=1,...,n$,
and then stretching linearly. Here, as usual, $u\wedge v$ denotes the unique branch point such that $[\rho,u\wedge v]=[\rho,u]\cap[\rho,v]$.

By construction,
$\dis(f_N) \tNo 0$ where
\begin{equation}
\label{e:dis}
   \dis(f)
 :=
    \sup_{x,y\in T} \bigl|r(x,y)-r'(f(x),f(y))\bigr|
\end{equation}
denotes the {\em distortion} of a map $f:(T,r)\to(T',r')$.
\label{Rem:003}
\end{remark}\sm

\begin{proof}[Proof of Lemma~\ref{lem:lemma42}]
Assume that $N$ is large enough, such that the function $f_N\colon T_N\to T$
from Remark~\ref{Rem:003}  is a tree homomorphism with $f_N(u_{N,k})=u_k$, $k\in\{1,2,\ldots,n\}$
and such that $\dis(f_N) \tNo 0$. We can therefore choose a metric $d$ on $T_N \uplus T$ extending $r_N$ and $r$
such that $d(x,f_N(x))\to 0$ for all $x\in T_N$ (compare, for example, \cite[Corollary~7.3.28]{BurBurIva01}).

Thus
$d_{\rm Pr}((f_N)_\ast \mu_N,\mu_N)\leq \sup_x d(x,f_N(x))\to 0$, as $N\to\infty$, and we obtain that
\begin{equation}
\label{e:001}
   d_{\rm Pr}\bigl((f_N)_\ast \mu_N,\mu\bigr)
 \leq
   d_{\rm Pr}\bigl((f_N)_\ast \mu_N,\mu_N\bigr)+d_{\rm Pr}\bigl(\mu_N,\mu\bigr)\tNo 0.
\end{equation}

Fix now $I\subseteq\{1,\ldots,n\}$ and define the subtree
\begin{equation}
	S := \rspan{\pi_I^n(\uvec)} \subseteq T.
\end{equation}
Because $S$ is closed in $T$, we have $\limsup_{N\to\infty} (f_N)_\ast\mu_N(S)\leq\mu(S)$
by the Portmanteau theorem (see Theorem~2.1 in \cite{Billingsley1999}). Because $\psi$ is increasing, this
implies
\begin{equation}\label{e:pointwiseineq}
	\limsup_{N\to\infty} \psi\((f_N)_*\mu_N(S)\) \leq \psi\(\mu(S)\).
\end{equation}
By assumption~(\ref{e:cond(i)}),
\begin{equation}
\label{e:002}
   \E\bigl[\psi\((f_N)_\ast\mu_N(S)\)\bigr]=\E\bigl[\psi\(\mu_N(\rspan{\pi_I^n(\uvec_N)})\)\bigr]\tNo
   \E\bigl[\psi\(\mu(S)\)\bigr].
\end{equation}
\eqref{e:002} and \eqref{e:pointwiseineq} together yield $\psi\((f_N)_\ast\mu_N(S)\)\tNo\psi\(\mu(S)\)$, almost surely.
Because $\psi$ is strictly increasing, also $(f_N)_\ast\mu_N(S)\to\mu(S)$.  Using once more the Portmanteau
theorem and closedness of $S$, we obtain that
\begin{equation}
   (f_N)_\ast\mu_N\restricted{S} \TNo \mu\restricted{S}.
\end{equation}
The inequality
\begin{equation} \label{e:003}
\begin{aligned}
   &d_{\rm pGw}\bigl(\pi_I^n\(\rspan{\uvec_N},\uvec_N,\mu_N\),\, \pi_I^n\(\rspan{\uvec},\uvec,\mu\)\bigr)
 \\
 &\leq
   d_{\rm Pr}\bigl(\mu_N{}\restricted{\rspan{\pi_I^n(\uvec_N)}},\,
   	(f_N)_\ast\mu_N\restricted{S}\bigr)
   + d_{\rm Pr}\bigl((f_N)_\ast\mu_N\restricted{S},\,
	   \mu\restricted{S}\bigr),
\end{aligned}
\end{equation}
then gives the {\rm sGw}-convergence.
\end{proof}\sm

As for the \pGwtopo, we define an associated set of test functions $\tilde\Phi:\Hn\to \R$ by
\begin{equation}\label{eq:phitilde}
	\tilde\Phi(T,\uvec,\mu) := \prod_{I \subseteq \{1,\ldots,n\}} \Phi_I\bigl( \pi^n_I(T,\uvec,\mu) \bigr),
\end{equation}
where the $\Phi_I$ are polynomials on $\H_{\#I}$.
Obviously, this class of test functions induces the \sGwtopo\ on $\Ht$, and together with the polynomials on $\Hn$, the \sGwtopo\ on $\Hn$.
We also define
\begin{equation}
\label{e:004}
   \Pit
 :=
   \bigl\{\prod_{I \subseteq \{1,\ldots,n\}} \Phipol_I\circ
   \pi^n_I:\,\Phi_I^{\gamma,m,\varphi}\in\Pi_{\#I}\bigr\}.
\end{equation}

\subsection{The LWV-topology}
\label{subsec:LWV-topo}
In this subsection we give the definition of \emph{bi-measure $\R$-trees} and equip the space of equivalence classes of bi-measure $\R$-trees with the
\emph{leaf-sampling weak vague topology}, in the following referred to as the \emph{LWV-topology}.

Given a rooted measure $\R$-tree $(T,\mu)$, denote by
\begin{equation}
\label{e:muskel}
   \muskel
  :=
    \bigcup_{v\in \supp(\mu)} {[\rho, v[} \cup \bigl\{v\in T:\;\mu\bigl(\{v\}\bigr)>0\bigr\}
\end{equation}
the \emph{$\mu$\nbd skeleton} of $(T,\mu)$, and by
\begin{equation}
\label{e:muleaf}
   \muleaf
  :=
    \rspan{\supp(\mu)} \setminus \muskel
\end{equation}
the set of \emph{$\mu$\nbd leaves} of $(T,\mu)$.

We call $(T,\mu,\nu)$ a (rooted) \emph{bi-measure $\R$-tree} if $(T,\mu)$ is a (rooted) measure $\R$-tree and $\nu$
is a ($\sigma$-finite) measure on $T$ which satisfies the following two conditions:
\begin{enumerate}
	\item\label{it:finiteonfinite} $\nu\([\rho, u]\)$ is $\mu$\nbd a.s.\ finite for $u\in T$,
	\item $\nu$ vanishes on the set of $\mu$\nbd leaves, i.e., $\nu(\muleaf)=0$.
\end{enumerate}
Note that \ref{it:finiteonfinite} implies that $\nu$ is finite on subtrees of $T$ with a finite number of leaves
sampled with $\mu$, a.s, and that $\nu\restricted{\muskel}$ is $\sigma$\nbd finite (because our definition of
measure $\R$-trees includes separability of $\supp(\mu)$).
In many interesting cases, however, $\nu$ is not locally finite.

\begin{definition}[The spaces $\Hbi$ and $\HK$]
Two bi-measure \Rtree s $(T,\mu, \nu)$ and\/ $(T', \mu', \nu')$ are called \emph{equivalent} if there exists
an isometry  $\phi\colon \rspan{\supp(\mu)} \to T'$ preserving the root and\/ $\mu$ and preserving $\nu$ on the
$\mu$\nbd skeleton, i.e., $\phi_\ast(\mu)=\mu'$ and\/ $\phi_\ast(\nu\restricted{\muskel}) = \nu'\restricted{\muskelp}$.
In particular, $(T, \mu, \nu)$ is equivalent to $(T, \mu, \nu\restricted{\muskel})$.

We denote by\/ $\mathbb H^{f,\sigma}$ the space of equivalence classes of (rooted) bi-measure\/ \Rtree s, and
by\/ $\HK:=\bset{(T, \mu, \nu) \in \Hbi}{\|\mu\|\le K}$, $K>0$, the subspace where the total mass of the sampling
measure is bounded by\/ $K$.
\label{Def:003}
\end{definition}\sm

Similar to the distance matrix distribution $\bnu^{(T,\mu)}$ introduced in (\ref{fct:rootedR}), which
characterizes $n$-pointed measure \Rtree s and is used to define the \pGwtopo, we want to characterize
bi-measure $\R$-trees by the so-called \emph{subtree-vector-distribution}. To introduce this, consider for
a given bi-measure \Rtree\ $(T,\mu,\nu)$ the function
\begin{equation}
\label{e:tauxn}
	\tauxn[(T,\mu,\nu)]\colon \left\{\begin{array}{rcl} \bigcup_{n\in\N} T^{n}&\to& \bigcup_{n\in\N}\Hn,
  \\
    (u_1,u_2,\ldots,u_n)&\mapsto&\bigl(\rspan{u_1,\ldots,u_n},(u_1,\ldots,u_n),\nu\bigr),
    \end{array}\right.
\end{equation}
which sends a vector of $n$ points in $T$ to the $n$-pointed \Rtree\ spanned by these points and
equipped with $\nu$, which we tacitly understand to be restricted to the appropriate space, i.e.\ $\rspan{u_1,...,u_n}$.
We also define the function
\begin{equation}
\label{e:taux}
	\taux[(T,\mu,\nu)]\colon \left\{\begin{array}{rcl} T^{\mathbb{N}}&\to&\prod_{n\in\mathbb{N}}\Hn,
  \\[1mm]
    \underline{u}&\mapsto&\bigl(\tauxn[(T,\mu,\nu)](u_1),\,\tauxn[(T,\mu,\nu)](u_1,u_2),\,\ldots\bigr),
    \end{array}\right.
\end{equation}
which sends a sequence of points to the sequence of pointed measure \Rtree s spanned and pointed by the first $1$,
$2$, etc.\ points and each of these is equipped with the appropriate restriction of $\nu$.
Note that $\tauxn[(T,\mu,\nu)]$ does not depend on the measure $\mu$ and is in general not continuous.

\begin{lemma}[Measurability]
	Equip\/ $\Hn$ with the $n$-pointed Gromov-weak topology, and\/ $\prod_{n\in\N} \Hn$ with the
	product topology. Then the function $\taux$ is measurable for all\/ $\smallx\in\Hbi$.
\label{Lem:001}
\end{lemma}\sm

\begin{proof}
It is enough to show that $\tauxn$ is measurable on $T^n$ for each $n\in\N$. Fix therefore $n\in\mathbb{N}$.

Since $\Hn$ is separable (Proposition~\ref{prop:pGwcharacterizations}), and the space of all polynomials induces the $n$-pointed Gromov-weak topology on
$\Hn$, it is enough to show
that $\Phi\circ \tauxn$ is measurable for every polynomial $\Phi$ (compare (\ref{e:polynomial})).
As for each $m\in\mathbb{N}$, $\vphi\in{\mathcal C}_b(\R_+^{m+n+1\choose 2})$ and $\smallx=(T,\mu,\nu)$,
\begin{equation}
\label{e:tauxmeas}
   \Phi^{m,\vphi}\circ\tauxn(\uvec)
 =
   \int\nu^{\otimes m}(\mathrm{d}\vvec)\,\mathds 1_{\{v_1,\ldots,v_m\in \rspanu\}}\,\varphi\bigl(R^T(\uvec,\vvec)\bigr),
\end{equation}
this follows from joint measurability of $(\uvec,\vvec)\mapsto\mathds 1_{\{v_1,\ldots,v_m\in \rspanu\}}(\varphi\circ R^T)(\uvec,\vvec)$.
\end{proof}\sm

We are now in a position to define the
\emph{subtree vector distribution}, $\varpi^\smallx$, of a bi-measure $\R$-tree $\smallx=(T,\mu,\nu)$ as
\begin{equation}
\label{e:subtreedist}
   \varpi^\smallx
  :=
    \|\mu\|\cdot (\taux)_\ast\((\mu^\circ)^{\otimes\mathbb{N}}\) \in{\mathcal M}_f\bigl(\prod_{n\in\mathbb{N}}\Hn\bigr).
\end{equation}

\begin{definition}[LWV-topology]
We say that a sequence\/ $(\smallx_N)_{N\in\mathbb{N}}$ converges to $\smallx$  in\/ $\Hbi$
in the leaf-sampling weak vague topology (\/{\rm LWV}-topology) if
the corresponding subtree vector distributions converge, i.e.,
\begin{equation}
\label{e:LWVtopology}
	\varpi^{\smallx_N}
		\TNo
	\varpi^\smallx,
\end{equation}
where convergence is weak convergence of finite measures on\/ $\prod_n (\Hn,{\rm pGw})$.
\label{def:LWV-topo}
\end{definition}\sm

\begin{remark}
	Obviously, $\HK$ is closed in $\Hbi$ with LWV-topology, $\Hbi=\bigcup_{K\in\N} \HK$, and for every
	compact set $\K\subseteq \Hbi$ there exists $K\in\N$ with $\K\subseteq \HK$.
\end{remark}\sm

\begin{rem}[Relation with Gromov-weak topology]\hspace{1cm}
\begin{enumerate}
	\item {\rm LWV}-convergence of $(T_N, \mu_N,\nu_N)_{N\in\mathbb{N}}$ implies Gromov-weak convergence of $(T_N, \mu_N)_{N\in\mathbb{N}}$.
	\item Gromov-weak convergence of $(T_N, \mu_N)_{N\in\mathbb{N}}$ does not imply {\rm LWV}-convergence\ of $(T_N, \mu_N,
		\mu_N)_{N\in\mathbb{N}}$ (compare Example~\ref{ex:disappear2}).\exend
\end{enumerate}
\label{Rem:004}
\end{rem}\sm

Recall from Definition~\ref{Def:002} and Definition~\ref{def:sGw}
the $n$-pointed Gromov-weak topology ($\rm pGw$) and the subtree Gromov-weak topology ($\rm sGw$), respectively.
Let $(U_{N,k})_{k\in\N}$ be an i.i.d.\ sequence of $\mu_N^\circ$\nbd distributed random variables,
and $\underline{U}^n_N:=(U_{N,1},\ldots, U_{N,n})$.
The definition of \LWVconv\ requires, in addition to convergence of $\|\mu_N\|$, the \emph{joint} convergence in law
with respect to the \pGwtopo\ of $\tauxNn(\underline{U}_N^n)$, $n\in\N$.
The next proposition shows that we can, on one hand, weaken this requirement to individual convergence of all
$\tauxNn(\underline{U}_N^n)$, and, on the other hand, strengthen it to require convergence in law with respect
to the \sGwtopo.

\begin{proposition}[Characterization of LWV-convergence]
	Consider a sequence of bi-measure \Rtree s $\smallx_N=(T_N,\mu_N,\nu_N) \in \Hbi$ and another bi-measure
	\Rtree\ $\smallx\in\mathbb H^{f,\sigma}$ such that\/ $\|\mu_N\|\to \|\mu\|$, as $N\to\infty$.
	The three following statements are equivalent:
\begin{enumerate}
	\item\label{it:pGw} $\smallx_N\toLWV\smallx$, as $N\to\infty$.
	\item\label{it:separate} For all $n\in\mathbb{N}$,
	\begin{equation}\label{eq:equiv:pGw}
		(\tauxNn)_\ast\bigl(\mu^\circ_N\bigr)^{\otimes n}
			\underset{N\rightarrow \infty}{\stackrel{\rm pGw}{\Longrightarrow}}
			(\tauxn)_\ast\bigl(\mu^\circ\bigr)^{\otimes n}.
	\end{equation}
	\item\label{it:sGw} Equipping\/ $\prod_{n\in\mathbb{N}}\Hn$ with the product topology $\prod(\mathrm{sGw})$,
	\begin{equation} \label{eq:equiv:PisGw}
		 \bigl(\tauxN\bigr)_\ast\bigl(\mu^\circ_N\bigr)^{\otimes\mathbb N}
			 \,\underset{N\rightarrow\infty}{\stackrel{\scriptscriptstyle\prod(\mathrm{sGw})}{\Longrightarrow}}\,
			 \bigl(\taux\bigr)_\ast\bigl(\mu^\circ\bigr)^{\otimes\mathbb N}.
	\end{equation}
\end{enumerate}
\label{prop:equivconvergence}
\end{proposition}\sm

\begin{proof}
First remark that \ref{it:sGw} $\Rightarrow$ \ref{it:pGw} $\Rightarrow$ \ref{it:separate} is straightforward.

We prove that \ref{it:separate} implies \ref{it:sGw}. Fix therefore $n\in\N$. By Skorohod's
representation theorem (Theorem~6.7 in \cite{Billingsley1999}), there exists a list
$\underline{U}^n=(U_1,\dots,U_n)$ of $n$ i.i.d.\ random variables with common distribution $\mu^\circ$ and
$\underline{U}^n_N=(U_{N,1},\dots,U_{N,n})$ i.i.d.\ random variables with distribution $\mu_N^\circ$ such that
\begin{equation}
\label{e:005}
   \tauxNn(\underline{U}^n_N)\underset{N\rightarrow \infty}{\topGw}  \tauxn(\underline{U}^n), \; \text{ almost surely.}
\end{equation}
In order to obtain \sGwconv, by Lemma~\ref{lem:lemma42}, it is sufficient to prove for all $I\subseteq\{1,\ldots,n\}$
that $\nu_N\(\rspan{\pi_I^n(\underline{U}^n_N)}\)$ converges weakly (as $\R_+$\nbd valued random variable) to
$\nu\(\rspan{\pi_I^n(\underline{U}^n)}\)$. Because $\pi_I^n(\underline{U}^n_N)$ has the same distribution as
$\underline{U}^{\#I}_N$, and similarly for $\underline{U}^n$ instead of $\underline{U}_N^n$, this follows from
\eqref{eq:equiv:pGw} for $n=\#I$, where we use that the total mass of an $n$-pointed measure \Rtree\ is
continuous in the \pGwtopo.
Finally, we conclude from Lemma~\ref{lem:lemma42} that
\begin{equation}
\label{e:008}
   \tauxNn(\underline{U}^n_N)\tosGw  \tauxn(\underline{U}^n), \; \text{as $N\to\infty$,  almost surely.}
\end{equation}
In particular, the one-dimensional marginals of $(\tauxN)_\ast\left(\mu^\circ_N\right)^{\otimes\mathbb N_0}$
converge as measures on $(\Hn, {\rm sGw})$. In order to obtain convergence of laws on the product space, we have
to show convergence of finite-dimensional marginals. This comes directly from the definition of {\rm sGw}-convergence.
\end{proof}\sm

We are now in a position to show that the subtree vector distribution characterizes bi-measure $\R$-trees uniquely.
\begin{proposition}[Reconstruction theorem for $\mathbb{H}^{f,\sigma}$]
	If $\smallx,\smallx'\in\mathbb{H}^{f,\sigma}$ are such that $\varpi^\smallx=\varpi^{\smallx'}$,
	then $\smallx=\smallx'$.
\label{prop:reconstruction}
\end{proposition}

\begin{proof}  Let $\smallx=(T,\mu,\nu),\smallx'=(T',\mu', \nu')\in\mathbb{H}^{f,\sigma}$ with
$\varpi^\smallx=\varpi^{\smallx'}$. It follows immediately that $\|\mu\|=\|\mu'\|$.
Assume w.l.o.g.\ that $\|\mu\|=\|\mu'\|=1$.\sm
	
We will first  adapt Vershik's proof of Gromov's reconstruction theorem for metric measure spaces
to show that $(T,\mu)=(T',\mu')$ (compare~\cite[$3\frac12.7$]{Gromov2000}).
Recall that a sequence $\uvec=(u_n)_{n\in\mathbb{N}}$ in $T$ is called \emph{$\mu$\nbd uniformly distributed} if
\begin{equation}\label{eq:uniformlydist}
	\tfrac{1}{n}\sum\nolimits_{i=1}^n\delta_{u_i}\Tno\mu,
\end{equation}
and note that, due to separability of $T$, $\mu^{\otimes \N}$\nbd almost every sequence is $\mu$\nbd uniformly distributed (see, for example, \cite[Theorem~11.4.1]{Dudley2002}).

Of course, the corresponding statement is also true for $\mu'$ instead of $\mu$, and
as $\varpi^{\smallx}=\varpi^{\smallx'}$, we can find a $\mu$\nbd uniformly distributed sequence
$\uvec=\folge{u}$ in $T$, and a $\mu'$\nbd uniformly distributed sequence $\uvec'=\folge{u'}$ in $T'$ with
$\taux(\uvec)=\taux[\smallx'](\uvec')$.
\sm
		
Put $u_0:=\rho$, $u_0':=\rho'$. Then $f(u_k) := u'_k$, for all $k\in\mathbb{N}_0$,  defines a root-preserving isometry from
$\{u_0,u_1,\ldots\}$ onto $\{u'_0,u'_1,\ldots\}$, which can be extended to an isometry (still denoted by $f$)
from $\rspan{\supp(\mu)}$ onto $\rspan{\supp(\mu')}$ (see Remark~\ref{rem:000}).
Because the sequences are uniformly distributed and $f_*$ is continuous,
\begin{equation}\begin{aligned}
	f_*(\mu) &= f_*\Bigl(\nlim \tfrac1n\sum_{i=1}^n\delta_{u_i}\Bigr)
		= \nlim f_*\Bigl(\tfrac1n\sum_{i=1}^n\delta_{u_i}\Bigr) \\
	 &= \nlim \tfrac1n\sum_{i=1}^n \delta_{u'_i} =  \mu'.
\end{aligned}\end{equation}
\sm

We still need to show that $f_\ast(\nu')=\nu$ (on the $\mu$-skeleton), or equivalently,
$f_\ast(\nu')(S)=\nu(S)$ for all finite trees $S\subseteq \muskel$.
		By definition of $\muskel$ and the fact that $\folge{u}$ is uniformly distributed, we have $S
		\subseteq \rspan{\uvec^n}$ for sufficiently large $n$.
		Because $(\rspan{\uvec^n}, \uvec^n, \nu)$ and $(\rspan{{\uvec'}^n}, {\uvec'}^n, \nu')$ are
		equivalent as $n$-pointed metric measure spaces,
		$f_\ast(\nu')\restricted{\rspan{\uvec^n}} = \nu\restricted{\rspan{\uvec^n}}$.
\end{proof}\sm

We can now immediately conclude that $\Hbi$ is separable and metrizable. We are not able to come up, however, with a complete metric.
``Polishness'' of the state space will not be used throughout the paper.
\begin{cor}[Separability \& metrizability]
	The space $\Hbi$ equipped with the {\rm LWV}-topology is separable and metrizable.
\label{cor:metrizable}
\end{cor}

\begin{proof}
	As the map which sends a bi-measure $\R$-tree to its subtree vector distribution is injective, we
	can identify $\Hbi$ with a subspace of ${\mathcal M}_f(\prod_{n\in\mathbb{N}}\Hn)$.
	$\Hn$ is separable, metrizable according to Proposition~\ref{prop:pGwcharacterizations}, hence the same
	holds for the countable product and the space of finite measures on it (with weak topology).
\end{proof}\sm

It is important to note that $\mu$ and $\nu$ play different r\^oles in the {\rm LWV}-topology, even if $\nu$
happens to be finite and $\mu$ is supported on the skeleton.  While the convergence is \emph{weak} with respect
to $\mu$, it is \emph{vague} with respect to $\nu$ in the sense that the total $\nu$-mass is not preserved under
convergence, but mass may get lost in the limit. We give two examples of this phenomenon.


\begin{figure}

  \begin{minipage}[b]{0.49\linewidth}
   \centering
\scalebox{1} 
{
\begin{pspicture}(0,-2.2929688)(2.8028126,2.2929688)
\definecolor{color3974}{rgb}{0.0,0.0,0.8}
\definecolor{color3409}{rgb}{1.0,0.0,0.2}
\psdots[dotsize=0.24,linecolor=color3974,fillstyle=solid,dotstyle=square](1.2609375,0.11453125)
\psline[linewidth=0.04cm](0.5209375,2.1345313)(0.515893,-2.1214974)
\psline[linewidth=0.04cm](0.5209375,0.13453124)(2.0209374,0.13453124)
\psdots[dotsize=0.14,fillstyle=solid,dotstyle=o](0.515893,-2.1214974)
\psdots[dotsize=0.14,fillstyle=solid,dotangle=-233.8279,dotstyle=o](1.2658525,0.1208301)
\psdots[dotsize=0.14,fillstyle=solid,dotangle=-128.35576,dotstyle=o](0.52378684,0.13139208)
\usefont{T1}{ptm}{m}{n}
\rput(0.8723438,-2.1154687){$\rho$}
\usefont{T1}{ptm}{m}{n}
\rput(0.22234374,0.16453125){$z$}
\psdots[dotsize=0.14,fillstyle=solid,dotangle=-128.35576,dotstyle=o](0.52378684,2.131392)
\psdots[dotsize=0.14,fillstyle=solid,dotangle=-233.8279,dotstyle=o](2.0658524,0.1208301)
\psdots[dotsize=0.2,linecolor=color3409,dotstyle=x](2.0609374,0.11453125)
\psdots[dotsize=0.2,linecolor=color3409,dotstyle=x](0.5209375,2.1345313)
\usefont{T1}{ptm}{m}{n}
\rput(0.85234374,2.1045313){$y_1$}
\usefont{T1}{ptm}{m}{n}
\rput(2.3123438,0.14453125){$y_2$}
\usefont{T1}{ptm}{m}{n}
\rput(1.2123437,0.38453126){$w$}
\end{pspicture}
}
\caption{}   \label{fig:disapears1}
  \end{minipage}
  \hfill
  \begin{minipage}[b]{0.49\linewidth}
   \centering

\scalebox{1} 
{
\begin{pspicture}(0,-2.3129687)(2.9428124,2.3129687)
\definecolor{color420}{rgb}{0.0,0.0,0.8}
\psdots[dotsize=0.24,linecolor=color420,fillstyle=solid,dotstyle=square](0.5409375,2.1145313)
\psdots[dotsize=0.24,linecolor=color420,fillstyle=solid,dotstyle=square](1.3409375,-1.1854688)
\psdots[dotsize=0.24,linecolor=color420,fillstyle=solid,dotstyle=square](2.0409374,0.09453125)
\psdots[dotsize=0.24,linecolor=color420,fillstyle=solid,dotstyle=square](1.8609375,0.9345313)
\psdots[dotsize=0.24,linecolor=color420,fillstyle=solid,dotstyle=square](1.3409375,1.4145312)
\psline[linewidth=0.04cm](0.5209375,0.11453125)(2.0209374,0.11453125)
\psdots[dotsize=0.14,fillstyle=solid,dotangle=-233.8279,dotstyle=o](2.0658524,0.1008301)
\psline[linewidth=0.04cm](0.50852025,0.12719882)(1.3285152,-1.12883)
\psdots[dotsize=0.14,fillstyle=solid,dotangle=-290.68948,dotstyle=o](1.3415958,-1.1739296)
\psline[linewidth=0.04cm](0.50852025,0.12692617)(1.3285152,1.382955)
\psdots[dotsize=0.14,fillstyle=solid,dotangle=-69.31051,dotstyle=o](1.3415958,1.4280546)
\psline[linewidth=0.04cm](0.534419,0.11711593)(1.8244572,0.88249254)
\psdots[dotsize=0.14,fillstyle=solid,dotangle=-95.49155,dotstyle=o](1.8560941,0.9171938)
\psline[linewidth=0.04cm](0.5209375,2.1145313)(0.515893,-2.1414974)
\psdots[dotsize=0.14,fillstyle=solid,dotstyle=o](0.515893,-2.1414974)
\psdots[dotsize=0.14,fillstyle=solid,dotangle=-128.35576,dotstyle=o](0.52378684,0.11139208)
\usefont{T1}{ptm}{m}{n}
\rput(0.8723438,-2.1354687){$\rho$}
\usefont{T1}{ptm}{m}{n}
\rput(0.22234374,0.14453125){$z$}
\psdots[dotsize=0.14,fillstyle=solid,dotangle=-128.35576,dotstyle=o](0.52378684,2.111392)
\psdots[dotsize=0.2,linecolor=red,dotstyle=x](0.5209375,2.1145313)
\usefont{T1}{ptm}{m}{n}
\rput(0.79234374,2.1245313){$y$}
\usefont{T1}{ptm}{m}{n}
\rput(1.6023438,1.6045313){$x_1$}
\psdots[dotsize=0.2,linecolor=red,dotstyle=x](1.3209375,1.4145312)
\psdots[dotsize=0.2,linecolor=red,dotstyle=x](1.8609375,0.91453123)
\psdots[dotsize=0.2,linecolor=red,dotstyle=x](2.0409374,0.09453125)
\psdots[dotsize=0.2,linecolor=red,dotstyle=x](1.3409375,-1.1854688)
\psdots[dotsize=0.04](1.7009375,-0.24546875)
\psdots[dotsize=0.04](1.5809375,-0.44546875)
\psdots[dotsize=0.04](1.4409375,-0.62546873)
\usefont{T1}{ptm}{m}{n}
\rput(2.2023437,1.0645312){$x_2$}
\usefont{T1}{ptm}{m}{n}
\rput(2.4623437,0.12453125){$x_3$}
\usefont{T1}{ptm}{m}{n}
\rput(1.7823437,-1.3154688){$x_N$}
\end{pspicture}
}
\caption{}   \label{fig:disapears2}
\end{minipage}
\caption*{The crosses $\times$ are $\mu$-masses and the squares $\square$ are $\nu$-masses.}
\end{figure}


\begin{example}
Consider the (finite) \Rtree\ shown in Figure~\ref{fig:disapears1}
and define the probability measures $\mu_N:=(1-\frac 1N)\delta_{y_1}+\frac 1N \delta_{y_2}$. Then $(T,\mu_N)$
converges Gromov-weakly to $\(\{\rho,y_1\},\delta_{y_1}\)$. We endow $(T,\mu_N)$ with a constant measure
$\nu:=\delta_w$, then $(T,\mu_N,\delta_w)$ converges in the {\rm LWV}-topology to
$(\{\rho,y_1\},\delta_{y_1},0)$.
\label{ex:disappear1}
\end{example}\sm

\begin{example}[Figure~\ref{fig:disapears2}]
We define a sequence of \Rtree s $$T_N:=\bigl\{\rho, z,y,x_1,x_2,\dots,x_N\bigr\}$$ shown in
Figure~\ref{fig:disapears2} where $r_N(\rho,z)=r_N(z,y)=1$ and $r(z,x_i)=\frac 1N$, for all $i=1,\dots,N$. We
define a probability measure $\mu_N$ on the leaves of $T_N$ by $\mu_N=\lambda\delta_y+(1-\lambda)\sum_{i}\frac
1N \delta_{x_i}$, then $(T_N,\mu_N)$ converges Gromov-weakly to
$(\{\rho,z,y\},\lambda\delta_y+(1-\lambda)\delta_z)$. If we endow this measure \Rtree\ with the measure $\nu_N=\mu_N$, then $(T_N,\mu_N,\nu_N)$ converges in the {\rm LWV}-topology to $(\{\rho,z,y\},\lambda\delta_y+(1-\lambda)\delta_z,\lambda\delta_y)$.
\label{ex:disappear2}
\end{example}\sm

\subsection{Convergence determining classes for the LWV-topology}\label{sec:convdet}

In this subsection, we introduce important classes of test functions and use them to obtain several convergence
results. Namely, we consider functions $\Psi=\Psipol\colon \mathbb H^{f,\sigma} \to \R$ of the form
\begin{equation}
	\Psi(\smallx):=\Psipol(\smallx) :=
		\gamma(\|\mu\|)\cdot\int_{T^n}\mu^{\otimes n}(\mathrm{d}\uvec)\, \Phi\(\tauxn(\uvec)\),
\end{equation}
where $\gamma\in{\mathcal C}_b(\R_+)$ and $\Phi\in\Cb(\Hn)$.

Recall $\Pin$ and $\Pit$ from (\ref{e:Pi}) and (\ref{e:004}).
As we will see later, the following subspaces of test functions are helpful in characterizing \LWVconv. Put
\begin{equation}
\label{e:F}
   {\mathcal F}
  :=
    \bset{\Psipol[1]}{\Phi\in\Pin},
\end{equation}
and
\begin{equation}
\label{e:tildeF1}
   \tilde{\mathcal F}^1
  :=
    \bset{\Psitpol[1]}{\tilde{\Phi}\in\Pit},
\end{equation}
and
\begin{equation}
\label{e:Fgamma}
   \tilde{{\mathcal F}}
  :=
    \bset{\Psitpol}{\tilde{\Phi}\in\Pit,\,\lim_{x\to\infty}x^k\gamma(x)=0\;\forall k\in\mathbb{N}}.
\end{equation}

\begin{lemma}[LWV-convergence via test functions]
Both\/ $\mathcal F$ and\/ $\tilde{\mathcal F}$ induce the {\rm LWV}-topology, i.e.,
for a sequence of bi-measure \Rtree s $\smallx_N\in \Hbi$ and another bi-measure \Rtree\
$\smallx\in\Hbi$,  the following statements are equivalent.
\begin{enumerate}
	\item\label{it:lwv} $\smallx_N  \toLWV \smallx$, as $N\to\infty$.
	\item\label{it:fconv} $\Psi(\smallx_N)\to \Psi(\smallx)$, as $N\to\infty$, for all $\Psi\in{\mathcal F}$.
	\item\label{it:ftconv} $\tilde\Psi(\smallx_N)\to \tilde\Psi(\smallx)$, as $N\to\infty$, for all $\tilde\Psi\in\tilde{\mathcal F}$.
\end{enumerate}
\label{lem:topinduce}
\end{lemma}

\begin{proof}
The equivalence of \ref{it:lwv} and \ref{it:fconv} is clear, as by Proposition \ref{prop:equivconvergence}, {\rm LWV}-convergence
is equivalent to the convergence of $\|\mu_N\|\to \|\mu\|$ together with
$\left< (\tauxNn)_\ast\left(\mu^\circ_N\right)^{\otimes n} , f\right> \to \left< \left(\tauxn\right) \!_\ast\left(\mu^\circ\right)^{\otimes n} , f\right>$, as $N\to\infty$,
for all $n\in \N$ and for a class of functions $f$ which determine the $n$\nbd pointed Gromov-weak convergence.
Moreover, by Proposition~\ref{prop:pGwcharacterizations}, $\Pin$ is such a convergence determining class.
As $\Psi(\smallx_N)=\left<(\tau_{\smallx_N})_\ast\left(\mu_N\right)^{\otimes n},\Phi\right>$, the claim follows. \sm
\comment{\begin{eqnarray*}
\smallx_N  \toLWV \smallx  & \Leftrightarrow & \left<  \left(\mu^\circ_N\right)^{\otimes n} ,
\Phi^\gamma\circ\tauxNn(\cdot)\right> \to \left< \left(\mu^\circ\right)^{\otimes n} , \Phi^\gamma\circ\tauxn(\cdot) \right>\; \forall n\in \N, \forall \Phi^\gamma\in \Pin \\
& \Leftrightarrow & \Psi(\smallx_N)\to \Psi(\smallx) \; \forall \Psi\in \mathcal F.
\end{eqnarray*}}

By  Proposition~\ref{prop:equivconvergence}, $\tilde{\mathcal F}$ contains only functions which are continuous
with respect to the \LWVtopo, and thus \ref{it:lwv} clearly implies \ref{it:ftconv}. To see that \ref{it:ftconv}
implies \ref{it:fconv}, note that for $\gamma(x):=e^{-x}$, convergence of
$\Psi^{\gamma,0,1}(\smallx_N)=\gamma\(\|\mu_N\|\)$ implies convergence of $\|\mu_N\|$. Hence convergence of
$\tilde\Psi(\smallx_N)$, for all $\tilde{\Psi}\in\tilde{\mathcal F}$, implies convergence of $\Psi(\smallx_N)$,
for all $\Psi\in{\mathcal F}$.
\end{proof}\sm

\begin{proposition}[Convergence determining classes] The following hold:
\begin{enumerate}
\item The class of test functions $\tilde{\mathcal F}$ is convergence determining on $\mathbb H^{f,\sigma}$.
\item The class of test functions $\tilde{\mathcal F}^1$ is convergence determining on $\HK$ for all $K>0$.
\end{enumerate}
\label{prop:convdet}
\end{proposition}

\begin{proof}
We apply Theorem~6 from \cite{BloKour2010}, a slight extension of Le~Cam's theorem (see \cite{LeCam57}) in the
separable, metrizable case: if a set of bounded real-valued functions is multiplicatively closed and induces a
separable, metrizable topology, then it is a convergence determining class with respect to this topology.
By Lemma~\ref{lem:topinduce}, $\tilde{\mathcal F}$ induces the {\rm LWV}-topology, which is separable,
metrizable by Corollary~\ref{cor:metrizable}. We therefore need to verify that if $\tilde\Psi_1$,
$\tilde\Psi_2\in \Ft$, then $\tilde\Psi_1 \cdot\tilde\Psi_2 \in \Ft$.
Let $\Psit_i=\Psi^{\gamma_i,n_i,\tilde\Phi_i}$ for some $n_i\in\N_0$, $\gamma_i\in\Cb(\R_+)$ with
$\lim_{x\to\infty}x^k\gamma_i(x)=0$, for all $k\in \N$, and $\tilde{\Phi}_i\in\Pit$, $i=1,2$. Then
\begin{equation}
\begin{aligned}
   &\Psi^{\gamma_1,n_1,\tilde\Phi_1}\cdot\Psi^{\gamma_2,n_2,\tilde\Phi_2}(\smallx)
   \\
 &=
   \bigl(\gamma_1\gamma_2\bigr)(\|\mu\|)\cdot\int_{T^{n_1+n_2}} \mu^{\otimes (n_1+n_2)}(\mathrm{d}\uvec_1,\mathrm{d}\uvec_2)\,
   \tilde\Phi_1\bigl( \rspan{\uvec_1} , \uvec_1, \nu \bigr)\tilde\Phi_2\bigl( \rspan{\uvec_2} , \uvec_2, \nu \bigr).
\end{aligned}
\end{equation}
For $\uvec=(\uvec_1,\uvec_2)$, let $\Phit\left( \rspan{\uvec} , \uvec, \nu \right) :=
\Phit_1\left( \rspan{\uvec_1}, \uvec_1, \nu \right) \cdot \Phit_2\left( \rspan{\uvec_2} , \uvec_2, \nu \right)$.
As $\uvec_1$ and $\uvec_2$ are sublists of $\uvec$, $\Phit\in\Pit$ and therefore $\tilde\Psi_1
\cdot\tilde\Psi_2 \in \Ft$.\sm

To get the second statement in the same way, note that functions $\Psitpol[1]\in \tilde{\mathcal F}^1$ are
bounded on $\mathbb H^{K,\sigma}$.
\end{proof}\sm

An important fact about the \LWVtopo\ is that Gromov-weak convergence of measure \Rtree s implies \LWVconv\ if
the trees are additionally equipped with their respective length measures (see Example~\ref{ex:lengthmeasure}
for a definition of length measure and Proposition~\ref{prop:lengthconv} for the statement).
We obtain the same also for a slightly more general class of measures. Given a family $(T_i, \mu_i)_{i\in I}$ of
measure \Rtree s, we say that a family $(\nu_i)_{i\in I}$ of measures on respective $T_i$
\emph{depends continuously on the distances} if, for all $n\in\N$, there exists a continuous mapping
$F_n\colon\R^{\binom{n+1}{2}}\to \Hn$, where $\Hn$ is endowed with the \pGwtopo, such that
\begin{equation}\label{eq:lengthdepend}
	\left(\rspanu , \uvec, \nu_i\right)=F_n\(R^{T_i}(\uvec)\), \;\; \forall \uvec \in T_i^n,\;\forall i\in I.
\end{equation}

\begin{example}[Length measure] The length measure, $\lambda_T$, on a separable $0$-hyperbolic and connected metric space $T$
generalizes the Lebesgue measure on $\R$ in an obvious way (compare~\cite{EvaPitWin2006}).
Recall the set of leaves of $T$ from (\ref{e:leaves}). The length measure can  be defined by the following two
requirements:
\begin{equation}
\label{e:length}
    \forall x,y\in T :\, \lambda_T([x,y])=r(x,y) \;\text { and }\; \lambda_T\({\rm Lf}(T)\)=0.
\end{equation}
Obviously, the family of length measures $(\lambda_T)_{T\in \{\text{\Rtree s}\}}$ depends continuously on the distances.
The same is true if we replace $\lambda_T$ by $\nu_T=f_T\cdot \lambda_T$, where $f_T$ is a density that depends
only on the height, i.e., $f_T(v):=h\(r(\rho, v)\)$ for a bounded measurable function $h$ (which does not depend
on $T$).
\label{ex:lengthmeasure}
\end{example}\sm

We can relax the continuity of the $F_n$, $n\in\N$, a little.
Let $(T, \mu)\in \H$. We say that a family $(\nu_i)_{i\in I}$ as above \emph{depends\/ \bnuTas\ continuously on
the distances} if it satisfies \eqref{eq:lengthdepend} with functions $F_n$ that are not necessarily continuous,
but where the set of discontinuity points is a null set with respect to the distance matrix distribution induced
by $(T, \mu)$, i.e.\ $(R^{T})_\ast\mu^{\otimes n}({\rm Discont}(F_n))=0$.

\begin{proposition}[LWV-convergence from Gromov-weak convergence]
	Consider a sequence\/ $(\smallx_N)_{N\in\mathbb{N}}:=(T_N,\mu_N,\nu_N)_{N\in\mathbb{N}}$ and\/
	$\smallx_\infty:=(T_\infty,\mu_\infty,\nu_\infty)$ in $\H^{f,\sigma}$ such that the measures
	$\nu_\infty,\nu_1,\nu_2,...$ depend\/ \bnuTinfas\ continuously on the distances.

	If\/  $(T_N,\mu_N)\toGwN (T_\infty,\mu_\infty)$, then
	\begin{equation}\label{e:033}
		(T_N,\mu_N,\nu_N)  \toLWVN  (T_\infty,\mu_\infty,\nu_\infty).
	\end{equation}
	In particular, the embedding defined by
	\begin{equation} \label{e:034}
		\begin{array}{ccc}
		\H&\to& \Hbi, \\ (T, \mu) &\mapsto& (T, \mu, \lambda_T),
		\end{array}
	\end{equation}
	where $\lambda_T$ is the length measure, is a homeomorphism onto its image.
\label{prop:lengthconv}
\end{proposition}\sm

\begin{proof}
	Given $n\in\N$, fix a function $F_n\colon\R_+^{n+1\choose 2}\to\Hn$ as in \eqref{eq:lengthdepend}, such
	that the set of discontinuity points of $F_n$ is a zero set with respect to $(R^{T_\infty})_*(\mu_\infty^{\otimes n})$.
	For $N\in\N\cup\{\infty\}$, let $\Uvec_N$ be a random vector in $T_N^n$ with distribution
	$(\mu_N^\circ)^{\otimes n}$. Then the assumed Gromov-weak convergence means that $\|\mu_N\| \to
	\|\mu_\infty\|$ and
	\begin{equation}
		R^{T_N}(\Uvec_N) \tolN R^{T_\infty}(\Uvec_\infty),
	\end{equation}
	where $\tol$ denotes convergence in law. By the continuous mapping theorem (see Theorem~5.1 in
	\cite{Billingsley1999}), we obtain
	\begin{equation}
		\(\rspanUN,\Uvec_N, \nu_N\) = F_n\(R^{T_N}(\Uvec_N)\) \tolN F_n\(R^{T_\infty}(\Uvec_\infty)\)
		= \(\rspanUN[\infty], \Uvec_\infty, \nu_\infty\).
	\end{equation}
	Using that $\(\rspanUN,\Uvec_N, \nu_N\)$ has law $(\tauxNn)_*\((\mu_N^\circ)^{\otimes n}\)$ for
	$N\in\N\cup\{\infty\}$, the claimed LWV-convergence $\smallx_N\toLWV \smallx_\infty$ now follows from
	Proposition~\ref{prop:equivconvergence}. That \eqref{e:034} defines a homeomorphism onto its image is
	now obvious, because the length measure depends continuously on the distances (see
	Example~\ref{ex:lengthmeasure}).
\end{proof}\sm

\begin{cor}[Sampling measure perturbation]
	Consider two sequences of bi-measure \Rtree s $\smallx_N^i:=(T_N,\mu^i_N,\nu_N)$, $i=1,2$ that differ by
	their sampling measures $\mu^1_N$ and $\mu^2_N$. Assume that $\smallx^1_N \toLWVN \smallx$, and that the
	pruning measures $(\nu_N)_{N\in\N}$ depend\/ \bnuTas\ continuously on the distances.
	If\/ $d_\mathrm{Pr}(\mu^1_N,\mu^2_N)\tNo 0$, then also  $\smallx^2_N\toLWVN\smallx$.
\label{cor:changemu}
\end{cor}\sm

\begin{proof}
	As  $\smallx^1_N\toLWVN\smallx$, implies that $(T_N,\mu^1_N)\tNo(T,\mu)$ in the Gw-topology, we get
	$d_{\mathrm{pGP}}((T_N,\mu^1_N),(T,\mu))\tNo 0$ by Proposition~\ref{prop:pGwcharacterizations}.
	Since $\mu^1_N$ and $\mu^2_N$ are defined on the same space $T_N$, the latter implies that also
	\begin{equation}\label{e:011}
		\lim_{N\rightarrow\infty}d_{\rm pGP}\bigl((T_N,\mu^2_N),(T,\mu)\bigr)=0,
	\end{equation}
	(compare (\ref{e:GPdistance})).
	Proposition \ref{prop:lengthconv} allows us to endow these metric measure spaces with the associated
	measures $\nu_N$ and some $\nu_\infty$ on $T$, defined by \eqref{eq:lengthdepend}.
	Because of uniqueness of {\rm LWV}\nbd limits, we have $(T,\mu,\nu_\infty)=(T,\mu,\nu)$.
\end{proof}\sm

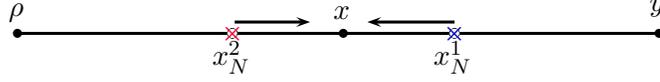
\begin{figure}
\scalebox{1} 
{
\begin{pspicture}(0,-0.538125)(9.262813,0.538125)
\definecolor{color86}{rgb}{1.0,0.0,0.2}
\definecolor{color100}{rgb}{0.0,0.0,0.8}
\psline[linewidth=0.04cm](9.0135145,0.00870204)(0.48356554,0.00870204)
\psdots[dotsize=0.12,dotangle=-90.0](8.980707,0.00870204)
\psdots[dotsize=0.12,dotangle=-90.0](0.45075804,0.00870204)
\psdots[dotsize=0.12,dotangle=-90.0](4.7813473,0.00870204)
\usefont{T1}{ptm}{m}{n}
\rput(3.2923439,-0.3103125){$x^2_N$}
\usefont{T1}{ptm}{m}{n}
\rput(4.762344,0.2896875){$x$}
\usefont{T1}{ptm}{m}{n}
\rput(8.952344,0.3496875){$y$}
\usefont{T1}{ptm}{m}{n}
\rput(0.43234375,0.2896875){$\rho$}
\usefont{T1}{ptm}{m}{n}
\rput(6.2323437,-0.3103125){$x^1_N$}
\psdots[dotsize=0.14,fillstyle=solid,dotstyle=o](3.30501,0.00870204)
\psdots[dotsize=0.2,linecolor=color86,dotstyle=x](3.30501,0.00870204)
\psdots[dotsize=0.14,fillstyle=solid,dotstyle=o](6.2576847,0.00870204)
\psdots[dotsize=0.2,linecolor=color100,dotstyle=x](6.2576847,0.00870204)
\psline[linewidth=0.04cm,arrowsize=0.05291667cm 2.0,arrowlength=1.4,arrowinset=0.4]{->}(3.3378177,0.1596875)(4.3409376,0.1596875)
\psline[linewidth=0.04cm,arrowsize=0.05291667cm 2.0,arrowlength=1.4,arrowinset=0.4]{<-}(5.1009374,0.1596875)(6.2609377,0.1596875)
\end{pspicture}
}
\caption{The tree $T_N$ with the two sequences $x^1_N$ and $x^2_N$ that converge to $x$.}\label{fig:counterexample}
\end{figure}

\begin{example}[Counterexample]
We cannot extend the result of Corollary~\ref{cor:changemu} to pruning measures which do not depend only on the distances.

As illustrated in Figure~\ref{fig:counterexample}, we consider a constant rooted metric space $T$ and two fixed points $x,y\in T$ such that $x\in[\rho,y]$.
We construct two sequences of points $(x^1_N)_{N\in\N}$ and $(x^2_N)_{N\in \N}$ that converge to $x$, the first
from above, the second from below; i.e.\ $x^1_N\in [x,y]$ and $x^2_N\in [\rho,x]$ for all $N\in\N$, and $r(x^i_N,x)\tNo 0$ for $i=1,2$. We then define the two sequences of measures $\mu^i_N:=\frac 12\delta_{x^i_N}+\frac 12\delta_y$ for $i=1,2$ and a constant measure $\nu_N=\nu=\delta_x$. Clearly,
$\smallx^1_N\toLWVN\smallx$ and $d_\mathrm{Pr}(\mu^1_N,\mu^2_N)\tNo 0$,
 but the sequence $(T,\mu^2_N,\nu)$ does not converge, since the subtree $[\rho,x^2_N]$ never contains the point
 $x$, except at the limit. Thus $([\rho,x_N^2],\{x_N^2\},\nu)$ does not converge pointed Gromov-weakly.
\end{example}\sm

\begin{lemma}[Sum of pruning measures]
	Let\/ $\smallx^i_N=(T_N,\mu_N,\nu^i_N)\in\Hbi$ with\/
		$(T_N,\mu_N,\nu^i_N) \toLWV \smallx^i=(T,\mu,\nu^i)\in \Hbi$, as $N\to\infty$,
	for $i=1,2$. If\/ $(\nu^1_N)_{N\in \N}$ depends\/ \bnuTas\ continuously on the distances, we obtain
	\begin{equation}
		(T_N,\mu_N,\nu^1_N+\nu^2_N) \toLWVN (T,\mu,\nu^1+\nu^2).
	\end{equation}
\label{lem:sum}
\end{lemma}\sm

\begin{proof}
Fix $n\in\N$. Because $(\nu^1_N)_{N\in \N}$ depends \bnuTas\ continuously on the distances, we can choose $F_n$ as in
	\eqref{eq:lengthdepend}. Let $\Uvec_N$, $\Uvec$ be random variables with distribution $\mun[N]$, $\mun$,
	respectively. By the \LWVconv\ and the Skorohod representation theorem, we can couple them such that
	$\tauxn[\smallx_N^2](\Uvec_N) \topGwN \tauxn[\smallx^2](\Uvec)$, a.s., which implies $R^{T_N}(\Uvec_N) \tNo R^T(\Uvec)$.
	Because $R^T(\Uvec)$ is a.s.\ a continuity point of $F_n$, we also have
	\begin{equation}
\label{e:012} \tauxn[\smallx_N^1](\Uvec_N)=F_n\circ R^{T_N}(\Uvec_N)
			\,\topGwN\, F_n\circ R^T(\Uvec)=\tauxn[\smallx^1](\Uvec), \,\text{ a.s.}
	\end{equation}

As explained in Remark~\ref{Rem:003}, we can define functions $f_N \colon \rspanUN \to \rspanU$ such that
	a.s.\ $f_N(\Uvec_N) = \Uvec$ for large enough $N$, $\dis(f_N) \tNo 0$, and $(f_{N})_\ast(\nu_N^i) \TNo \nu^i$.
	Then also $f_{N*}(\nu_N^1+\nu_N^2) \TNo \nu^1 + \nu^2$, which implies
	$\(\rspanUN, \Uvec_N, \nu_N^1+\nu_N^2\) \topGwN \(\rspanU, \Uvec, \nu^1+\nu^2\)$, a.s.
	By Proposition~\ref{prop:equivconvergence}, this implies the claimed \LWVconv.
\end{proof}\sm

\begin{remark}[Assumption on  \bnuTas\ continuity is important]
	In Lemma~\ref{lem:sum}, we cannot drop the assumption that one of the measures depends\/ \bnuTas\ continuously
	on the distances, because then we cannot use the same coupling of $\Uvec_N$ to get almost sure convergence of
	$\tauxn[\smallx_N^i](\Uvec_N)$ for $i=1$ and for $i=2$.
\label{Rem:014}
\end{remark}\sm

If we get {\rm LWV}-convergence of a sequence of bi-measure \Rtree s, the following lemma asserts that
the limit is stable under a small perturbation of $\nu_N$ in a certain sense.

\begin{lemma}[Pruning measure perturbation] 
Consider two sequences of bi-measure \Rtree s $\smallx_N^i:=(T_N,\mu_N,\nu^i_N)$, $i=1,2$ that
differ by their pruning measures $\nu^1_N$ and\/ $\nu^2_N$. If the two pruning measures are Prohorov merging on subtrees sampled
by\/ $\mu_N^{\otimes n}$, i.e.,
\begin{equation}\label{lem:changenu:equiv}
\lim_{N\rightarrow \infty}d_\mathrm{Pr}\({\nu^1_N}\restricted{\rspan{U_N^n}},\,{\nu^2_N}\restricted{\rspan{U_N^n}}\)=0,\;\; \text{$\mu^{\otimes
n}_N$-a.s. },\; \forall n\in \mathbb N,
\end{equation}
then
$\smallx^1_N \toLWVN \smallx$, for some $\smallx=(T,\mu,\nu)$, implies  $\smallx^2_N \toLWVN \smallx$.
\comment{	\[(T_N,\mu_N,\nu^2_N)\underset{N\rightarrow \infty}{\stackrel{{\rm LWV}}{\longrightarrow}} (T,\mu,\nu).\]}
\label{lem:changenu}
\end{lemma}

\begin{proof}
Let $\underline{U}_N$ and $\underline{U}$ be sequences of independent $\mu_N$- and $\mu$-distributed random variables
in $T_N$ and $T$, respectively. Because $\nu^1$ and $\nu^2$ are defined on the same measure \Rtree, the Prohorov distance in
(\ref{lem:changenu:equiv}) is an upper bound for the {\rm pGP}-distance, and we obtain
\begin{equation}
\begin{aligned}
   &d_{\rm pGP}\left( \tauxn[\smallx_N^2](\underline{U}_N^n), \tauxn(\underline{U}^n)\right)
    \\
   &\leq
    d_{\rm Pr}\left({\nu^1_N}\restricted{\rspan{\underline{U}_N^n}},{\nu^2_N}\restricted{\rspan{\underline{U}_N^n}}\right) + d_{\rm pGP}\left( \tauxn[\smallx_N^1](\underline{U}_N^n), \tauxn(\underline{U}^n)\right) \tNo 0,
\end{aligned}
\end{equation}
almost surely, for all $n\in \N$. This implies
$(\tauxn[\smallx_N^2])_\ast\left(\mu^\circ_N\right)^{\otimes n}\stackrel{\rm pGw}{\Longrightarrow}
\left(\tauxn\right)_\ast\left(\mu^\circ\right)^{\otimes n} $ for all $n\in \N$, and
Proposition~\ref{prop:equivconvergence} gives the {\rm LWV}-convergence.
\end{proof}

We conclude this section by giving a simple, sufficient (but far from necessary) condition for relative
compactness of a set $\K \subseteq\Hbi$. Assume that for all $\smallx'=(T', \mu', \nu')\in \K$,
there is an isometric embedding of $T'$ into some common \Rtree\ $T$, and there are measures $\mu$ and $\nu$ on $T$
dominating all the (push forwards of) $\mu'$ and $\nu'$, respectively. Further assume that $\smallx:=(T, \mu, \nu) \in
\Hbi$. In other words,
\begin{equation}\label{eq:Sx}
	\K\subseteq \Sx:=\bset{(T, \mu', \nu')\in\Hbi}{\mu'\le\mu,\;\nu'\le\nu}.
\end{equation}
Then $\K$ is relatively compact, as the following lemma shows.

\begin{lemma}[Compactness of $\Sx$]
	Let\/ $\smallx=(T, \mu, \nu)\in\Hbi$.  Then\/ $\Sx$, defined in \eqref{eq:Sx}, is compact in the\/ \LWVtopo.
\end{lemma}
\begin{proof}
	Consider measures $\mu_N\le \mu$, $\nu_N\le\nu$, $N\in\N$.
	We have to find a subsequence of $\smallx_N:=(T, \mu_N, \nu_N)$ that converges in $\Sx$.
	Fix finite subtrees $T_n\subseteq T$, $n\in\N$, with $T_n\subseteq T_{n+1}$ and
	$\bigcup_{n\in\N} T_n \supseteq \muskel$.

	Because the family $(\mu_N)_{N\in\N}$ is uniformly $\sigma$\nbd additive and norm bounded, there exists a setwise
	convergent subsequence (\cite[Thm.~4.7.25]{BogachevI2007}).
	Assume w.l.o.g.\ that there is $\mu_\infty\in \CM_f(T)$ with $\mu_N(A) \tNo \mu_\infty(A)$ for
	all measurable $A\subseteq T$. Similarly, using Cantor's diagonalization argument, we may assume that
	$\nu_N\restricted{T_n}$ converges setwise to some $\nuh_n\in\CM_f(T_n)$, for every $n\in\N$. Define
	\begin{equation}
   \label{e:035}
		\nu_\infty(A) := \sup_{n\in\N} \nuh_n\(T_n\cap A\cap \measureskel T{\mu_\infty}\).
	\end{equation}
	
Because $\nuh_n\restricted{T_{n-1}} = \nuh_{n-1}$, we can easily check that $\nu_\infty$ is a measure on $T$
	and $\smallx_\infty:=(T, \mu_\infty, \nu_\infty)\in\Sx$.
	Furthermore, for measurable
	$A\subseteq \measureskel{T}{\mu_\infty}\subseteq\muskel\subseteq \bigcup_{n\in\N} T_n$, we obtain
	\begin{equation}
		\nu_\infty(A) = \sup_{n\in\N} \lim_{N\to\infty} \nu_N(A\cap T_n)
		\left\{\begin{array}{l}
				\D \le \liminf_{N\to\infty} \nu_N(A) \\
				\D \ge \smash{\limsup_{N\to\infty} \nu_N(A) - \sup_{n\in\N} \nu(A\setminus T_n)}
		\end{array}\right.
	\end{equation}
	Using $A\subseteq\bigcup_{n\in\N} T_n$, this implies
	\begin{equation}    \label{eq:strongconv}
		\nu_\infty(A) = \lim_{N\to\infty}\nu_N(A).
	\end{equation}
	We shall show that $\smallx_N \toLWVN \smallx_\infty$. By Lemma~\ref{lem:topinduce}, it is enough to show that
	$\Psi(\smallx_N)\to\Psi(\smallx_\infty)$ for all $\Psi\in\F$.  Let
	\begin{equation}
		G
 :=
    \big\{\uvec\in T^n:\,\nu\(\rspanu\setminus\measureskel{T}{\mu_\infty}\) = 0\big\}.
	\end{equation}
	Fix $\Psi=\Psi^{n,\Phi}\in{\mathcal F}$. Then \eqref{eq:strongconv} implies
	\begin{equation}\label{eq:ptwiseconv}
		\Phi\circ \tauxNn(\uvec) \toN \Phi\circ\tauxn[\smallx_\infty](\uvec)
			\quad\forall \uvec\in G,
	\end{equation}
	and with $B:=T^n\setminus G$ we estimate
	\begin{equation}
	\begin{aligned}
		\bigl|\Psi(\smallx_N) - \Psi(\smallx_\infty)\bigr| \le&\;  \mun[N](B)2\|\Phi\|_\infty
			+ \plainintsetmu[N]{G}{\!|\Phi\circ\tauxNn - \Phi\circ\tauxNn[\infty]|}\\
		& + \plainint{|\Phi\circ\tauxNn[\infty]|}{(\mun[N] - \mun[\infty])}.
	\end{aligned}
	\end{equation}
	The last term converges to zero because of the setwise convergence of $\mu_N$ to $\mu_\infty$, and the
	second term is bounded by
	$\plainintsetmu{G}{|\Phi\circ\tauxNn - \Phi\circ\tauxNn[\infty]|}$,
	which converges to zero according to \eqref{eq:ptwiseconv}, using the dominated convergence theorem.

	For every $(u_1,\ldots,u_n) \in \supp(\mu_\infty)^n \setminus G$, there is an index
	$k\in\{1,\ldots,n\}$ with $u_k\in \At(\nu)\setminus \At(\mu_\infty)$, where $\At$ denotes the set of
	atoms of a measure. Because $\At(\nu)$ is countable, this implies that $B$ is a
	$\mu_\infty$-null set. Again using setwise convergence of $\mu_N$, we obtain
		\[ \lim_{N\to\infty} \mun[N](B) = \mun[\infty](B) = 0. \qedhere \]
\end{proof}\sm

\section{The Pruning Process}
\label{S:pruning}
In this section, we present the construction of the bi-measure valued pruning process, $(X_t)_{t\ge 0}$.
In Subsection~\ref{sec:pruningprocess}, we carry out an explicit construction given a realization of the Poisson
point process which gives rise to a c\`adl\`ag path. We continue the construction in
Subsection~\ref{sec:strongMarkov} by adding randomness and establishing that the stochastic process obtained
this way has the strong Markov property.
In Subsection~\ref{sec:continuity}, we establish the Feller property from which we can conclude that the law of
the pruning process on Skorohod space is weakly continuous in the initial distribution on bi-measure \Rtree s.
Finally, in Subsection~\ref{sec:generator} we give an analytic characterization via the infinitesimal generator.

\subsection{Getting the construction started: pruning moves}
\label{sec:pruningprocess}
It is convenient to introduce randomness later and work initially in a setting where the {\em cut times} and {\em cut
points} are fixed. Given a bi-measure $\R$-tree, $(T,\mu, \nu)\in\Hbi$, consider a subset
$\pi\subseteq \mathbb{R}_+\times T$.
Although $\pi$ is associated with a particular class representative, it corresponds, of course, to a similar set
for any representative of the same equivalence class by mapping across using the appropriate root invariant
isometry. Then the set of cut points up to time $t$ is the projection of $\pi\cap\([0,t]\times T\)$ onto the
tree, i.e.
\begin{equation}
	\pi_t := \bset{v\in T}{\exists s\le t: (s,v)\in\pi}.
\end{equation}
For every $v\in T$, the \emph{tree pruned at $v$} is defined by
\begin{equation} \label{e:prun}
	T^v := \bset{w\in T}{v\notin [\rho , w ]}.
\end{equation}
The pruned tree at the set $\pi_t\subseteq T$, $T^{\pi_t}$, is the intersection of the trees $T^v$ pruned at
$v\in\pi_t$, i.e.,
\begin{equation} \label{e:Tpi}
	T^{\pi_t} := \bigcap_{v\in \pi_t} T^v.
\end{equation}
We equip the pruned tree $T^{\pi_t}$ with the restrictions of the measures $\mu$ and $\nu$. As always, we write
$(T^{\pi_t}, \mu, \nu)$ instead of $(T^{\pi_t}, \mu\restricted{T^{\pi_t}}, \nu\restricted{T^{\pi_t}})$ and
easily verify $(T^{\pi_t}, \mu, \nu)\in\Hbi$.

\begin{lemma}[C\`adl\`ag paths]
	Fix\/ $\smallx=(T,\mu,\nu)\in\Hbi$ and a set\/ $\pi\subseteq\R_+\times T$.
	The map\/ $t\mapsto \smallx_t := (T^{\pi_t}, \mu, \nu)$ is c\`adl\`ag with respect to the\/ \LWVtopo.
\label{lem:cadlag}
\end{lemma}

\begin{proof}
Let $0<s<t$. As $T^{\pi_t} \subseteq T^{\pi_s}$, we obtain for all $\Psi=\Psipol[1]\in\F$,
\begin{equation}
\label{e:018}
\begin{aligned}
   \big|\Psi(\smallx_s)-\Psi(\smallx_t)\big|
 &=
   \Big|\int_{(T^{\pi_s})^n\setminus(T^{\pi_t})^n} \Phi\circ\tauxn\;\mathrm{d}\mu^{\otimes n}\Big|
   \\
 &\leq
   \|\Phi\|_\infty\cdot\mu^{\otimes n}\bigl((T^{\pi_s})^n\setminus(T^{\pi_t})^n\bigr)
   \\
 &\leq
   \|\Phi\|_\infty\cdot n\cdot\|\mu\|^{n-1}\cdot\mu(T^{\pi_s}\setminus T^{\pi_t}).
\end{aligned}
\end{equation}
For fixed $s$, $\bigcap_{t>s} T^{\pi_s}\setminus T^{\pi_t}=\emptyset$, which implies that $\mu(T^{\pi_s}\setminus
T^{\pi_t})\to 0$, as $t\to s$ from the right. Because $\F$ induces the LWV-topology, this implies \emph{right continuity}. \sm

To construct the \emph{left limit}, define  $T_{t-}:=\cap_{0\le s<t} T^{\pi_s}\supseteq T^{\pi_t}$ for each $t> 0$,
and define $\smally_t:=(T_{t-}, \mu,\nu)$, which is obviously an element of $\Hbi$. Similarly as before, for
all $0<s<t$ and $\Psi\in{\mathcal F}$, there exists a constant $C=C^{\Psi}$ such that
\begin{equation}
\label{e:019}
   \big|\Psi(\smallx_s)-\Psi(\smally_t)\big| \leq C\cdot\mu\bigl(T^{\pi_s}\setminus T_{t-}\bigr).
\end{equation}
As, for fixed $t$, $\bigcap_{s<t}T^{\pi_s}\setminus T_{t-}=\emptyset$, $\smally_t$ is indeed the left limit.
\end{proof}\sm


\subsection{Continuing the construction: adding randomness}
\label{sec:strongMarkov}
In this subsection we define, given a bi-measure $\R$-tree $\smallx=(T,\mu,\nu)$, the pruning process of
$\smallx$, where $\pi$ is now the (random) Poisson point measure with intensity $\lambda\otimes\nu$ on
$\R_+\times T$. Here, we
identify an atomic measure $\mathfrak{m}$ on $T$ with the set $\At(\mathfrak{m})$ of its atoms and define
\begin{equation}
	T^\mathfrak{m} := T^{\At(\mathfrak m)} = \bigcap_{v\in \At(\mathfrak m)} T^v.
\end{equation}

\begin{definition}[The pruning process]
Fix a bi-measure \Rtree\, $\smallx:=(T,\mu,\nu)\in \Hbi$. Let $\pi^\smallx$ be the Poisson point measure on
$\R_+\times T$ with intensity measure $\lambda\otimes \nu$, where $\lambda$ is the Lebesgue measure on  $\R_+$.
We define \emph{the pruning process}, $X:=(X_t)_{t\geq 0}$, as the bi-measure \Rtree-valued process obtained by
pruning $X_0:=\smallx$ at the points of the Poisson point process
$\pi_t(\boldsymbol{\cdot}):=\pi^\smallx_t(\boldsymbol{\cdot}):=\pi^\smallx\([0,t]\times\boldsymbol{\cdot}\)$, i.e.,
\begin{equation}
   X_t := \bigl(T^{\pi_t},\mu,\nu\bigr) :=
   \bigl(T^{\pi_t},\mu\restricted{T^{\pi_t}},\nu\restricted{T^{\pi_t}}\bigr).
\end{equation}
$\mathbb E^\smallx$, or $\mathbb E$ if there is no confusion, denotes the distribution of the process $X$ starting from $X_0=\smallx$.
\end{definition}\sm

\begin{lemma}[Strong Markov property]
The pruning process\/ $X$ is a strong Markov process.
\label{Lem:002}
\end{lemma}

\begin{proof}
Denote by $(\A_t)_{t\geq 0}$ the filtration generated by the Poisson point process $(\pi_t)_{t\geq0}$.
Note that $X$ is adapted to this filtration.
Using the strong Markov property of the Poisson process, we get for every $t\ge0$, stopping time $\sigma$, and
$\uvec\in T^n$, $n\in\N$,
\begin{equation}
	\mathbb P\(\pi_{\sigma+t}(\rspan{\uvec})=0 \bigm| \A_\sigma\)
		= \mathds 1_{\{\pi_\sigma(\rspan{\uvec})=0\}}\mathbb P\(\pi_t(\rspan{\uvec})=0\).
\end{equation}
For every $\tilde\Psi=\Psitpol[1]\in\Ft^1$, this implies
\begin{equation}\begin{aligned}
	\E\left[ \tilde\Psi(X_{\sigma+t}) \Bigm| \A_\sigma\right]
		&= \int_{T^n} \mu^{\otimes n}(\mathrm{d}\uvec)\,
			\mathbb P\(\pi_{\sigma+t}(\rspan{\uvec})=0 \bigm| \A_\sigma\)\cdot \Phit(\tauxn(\uvec))\\
		&= \int_{(T^{\pi_\sigma})^n}\mu^{\otimes n}(\mathrm{d}\uvec)\, e^{-t\nu(\rspanu)} \cdot
			\Phit\(\tauxn(\uvec)\).
\end{aligned}\end{equation}
On the other hand, we also have
\begin{equation}
	\E^{X_\sigma}\left[ \tilde\Psi(X_t)\right] = \int_{(T^{\pi_\sigma})^n}\mu^{\otimes n}(\mathrm{d}\uvec)\,
		e^{-t\nu(\rspanu)}\cdot \tilde\Phi\(\tauxn(\uvec)\).
\end{equation}
Because $X_t\in\H^{\|\mu\|,\sigma}$, for all $t\ge 0$, and $\tilde{\mathcal F}^1$ is a separating class on this space, we
obtain the strong Markov property.
\end{proof}\sm

\subsection{Continuity of the pruning process}
\label{sec:continuity}
In this subsection we show that the law of $X_t$ under $\mathbb{P}^\smallx$ is weakly continuous in the
initial value $\smallx$ for each $t\ge 0$. This property is sometimes referred to as the \emph{Feller
property} of the corresponding semigroup $(S_t)_{t\ge 0}$, although this terminology is often restricted to
the case of a locally compact state space and transition operators that map the space
of continuous functions that vanish at infinity into itself. In the latter, more restrictive case, the Feller
property implies that the law of the whole process (as random variable on Skorohod space) depends continuously
on the initial value. If $S_t$ maps only $\Cb$ into itself, this is no longer the case in general, and one needs
an extra argument. The pruning process $(X_t)_{t\ge 0}$, however, does depend continuously on the initial
condition (Theorem~\ref{theo:main}).

Let $(S_t)_{t\ge0}$ be the semi-group associated to the pruning process $(X_t)_{t\ge 0}$, i.e.\
for $t\ge 0$ and a bounded measurable function $G:\mathbb{H}^{f,\sigma}\to\mathbb{R}$,
\begin{equation}
\label{e:semigroup}
     S_tG(\smallx):=\mathbb{E}^\smallx\bigl[G(X_t)\bigr].
\end{equation}

\begin{proposition}[Feller continuity]
The process $X:=(X_t)_{t\geq0}$ is Feller continuous, i.e., $S_t\(\Cb(\Hbi)\) \subseteq \Cb(\Hbi)$.
\label{prop:feller}
\end{proposition}

\begin{proof}
Consider the convergence of bi-measure \Rtree s $\smallx_N \toLWV \smallx$. Write $K:=\sup\{\|\mu_N\|,N\in \N\}$,
then the sequence converges in $\H^{K,\sigma}$. Because $\tilde{\mathcal F}^1$ is convergence determining on
$\H^{K,\sigma}$ (see Proposition~\ref{prop:convdet}), it is enough to prove for all $\tilde\Psi\in \tilde{\mathcal F}^1$, $t>0$
that
\begin{equation} \label{e:022}
	\E^{\smallx_N}\bigl[\tilde\Psi(X_t)\bigr]\tNo \E^{\smallx}\bigl[\tilde\Psi(X_t)\bigr].
\end{equation}
Fix therefore $\tilde\Psi=\Psitpol[1]\in\tilde{\mathcal F}^1$. Then
\begin{equation} \label{e:023}
\begin{aligned}
   \E^{\smallx_N}\bigl[\tilde\Psi(X_t)\bigr]
 &=
   \E^{\smallx_N}\Bigl[\int_{(T_N^{\pi_t})^n}\tilde\Phi \circ\tauxNn\;\mathrm{d}\mu_N^{\otimes n}\Bigr]
   \\
 &=
   \int_{T_N^n}\mu_N^{\otimes n}(\mathrm{d}\uvec)\,\mathbb P\bigl(\pi_t^{\smallx_N}(\rspan{\uvec})=0\bigr)\cdot\tilde\Phi\bigl(\rspan {\uvec},\uvec,\nu_N\bigr).
\end{aligned}
\end{equation}
Using $\mathbb P\left(\pi_t^{\smallx_N}(\rspan{\uvec})=0\right) = \exp(-t\nu_N(\rspan{\uvec}))$, we see that
$\E^{\smallx_N}\left[\tilde\Psi(X_t)\right]=\tilde\Psi'(\smallx_N)$ for some  $\tilde\Psi'\in \tilde{\mathcal
F}^1$. The convergence follows therefore from the \LWVconv\ of $(\smallx_N)_{N\in\mathbb{N}}$.
\end{proof}\sm

Consider a separable, metrizable space $E$ and a 
contraction semigroup $S=(S_t)_{t\ge 0}$ on $\Cb(E)$.
We define
\begin{equation}\label{eq:domS}
	\Dom(S) := \big\{f\in\Cb(E):\,\lim_{t\to 0}\|S_t f - f\|_\infty= 0\big\}.
\end{equation}
Note that $\Dom(S)$ is uniformly closed, $S_t$ maps $\Dom(S)$ into itself, and the restriction of $(S_t)_{t\ge 0}$
to $\Dom(S)$ is a strongly continuous contraction semigroup. In particular, the restricted semigroup has a
generator $\Omega_S\colon \Dom(\Omega_S) \to \Dom(S)$ with dense domain $\Dom(\Omega_S)\subseteq \Dom(S)$.

\begin{lemma}\label{lem:Skorohodcont}
	Let\/ $E$ be a separable, metrizable space, and\/ $Y^x=(Y^x_t)_{t\ge0}$, $x\in E$, an\/ $E$\nbd valued,
	Feller-continuous (time-homogeneous) Markov process with c\`adl\`ag paths and semigroup\/
	$S=(S_t)_{t\ge 0}$ on\/ $\Cb(E)$.
	Assume that there is a set\/ $\G\subseteq\Dom(S)$ that is multiplicatively closed and induces the topology of\/ $E$.
	Then the map
	\begin{equation} \label{e:initialstate}
	\begin{array}{ccc}
		\P(E) &\to&     \P\(D_E(\R_+)\), \\
		\eta  &\mapsto& \law(Y^\eta)
	\end{array}
	\end{equation}
	is continuous, where\/ $D_E(\R_+)$ is the space of c\`adl\`ag paths with Skorohod topology, $\law$ is
	the law of a process, and\/ $Y^\eta$ is the process with initial condition\/ $\law(Y^\eta_0)=\eta$,
	i.e., $\law(Y^\eta)=\integral{\law(Y^x)}{\eta}{x}$.
\end{lemma}\sm

\begin{proof}
	It is sufficient to prove that $\law(Y^{x_N}) \TNo \law(Y^x)$ for every convergent sequence $x_N\tNo x$ in $E$.
	Because $\G$  induces the topology of $E$, it strongly separates points (see Lemma~1 in \cite{BloKour2010}).
	According to Theorem~10 of \cite{BloKour2010}, it is therefore enough to prove that for all $f_1,\ldots,f_k\in \G$,
	\begin{equation}
		\(f_1(Y^{x_N}_t), \ldots, f_k(Y^{x_N}_t) \)_{t\ge0} \tolNs
			\(f_1(Y^x_t), \ldots, f_k(Y^x_t) \)_{t\ge0}
	\end{equation}
	in Skorohod space as $\R^k$\nbd valued processes.
	The assumed Feller continuity implies f.d.d.\ convergence, hence it is enough to prove tightness.

	To this end, we apply Theorem~3.9.4 of \cite{EthierKurtz86}.
	The linear span $C_a:=\linhull(\G)$ of $\G$ is an algebra contained in $\Dom(S)$, and the domain
	$\Dom(\Omega_S)$ of the generator $\Omega_S$ of $S$ is dense in $\Dom(S)$. For every $f\in
	\Dom(\Omega_S)$, we define $Z^N_t := \Omega_Sf(Y^{x_N}_t)$. Then the following hold:
	\begin{enumerate}
	\item\label{it:martingale} The processes
		$\D\Bigl( f(Y^{x_N}_t)-\int_0^t Z^N_s \;\mathrm{d}s \Bigr)_{t\ge 0} $
		are martingales.
	\item\label{it:bounded} For all $T\ge 0$,
		$\D\sup_{N\in\N}\E\bigl[\,\essup_{0\leq t\leq T}|Z^N_t|\,\bigr] \le \|\Omega_S f\|_\infty<\infty$.
	\end{enumerate}
	Now tightness of the processes
	$\(f_1(Y^{x_N}_t), \ldots, f_k(Y^{x_N}_t) \)_{t\ge0}$, $N\in\N$, for every fixed $f_1,\ldots,f_k\in
	C_a\supseteq\G$ follows from \cite[Thm.~3.9.4]{EthierKurtz86}.
\end{proof}\sm

\begin{theorem}[Continuity in the initial distribution]
	The law of\/ $X$ on the Skorohod space depends continuously on the initial condition.
\label{theo:main}
\end{theorem}
\begin{proof}
	It is sufficient to prove continuity for deterministic initial conditions.
	Every convergent sequence $\smallx_N\toLWVN\smallx$ in $\Hbi$ is contained in $\HK$ for some $K>0$, and
	the pruning process stays a.s.\ in that subspace. We verify the conditions of Lemma~\ref{lem:Skorohodcont}
	for the $\HK$\nbd valued pruning process. It has c\`adl\`ag paths (Lemma~\ref{lem:cadlag}), is
	Feller-continuous (Proposition~\ref{prop:feller}), and $\Ft^1\subseteq \Cb(\HK)$ is multiplicatively
	closed and induces the \LWVtopo. It remains to show that $\Ft^1 \subseteq \Dom(S)$, where $S$ is the
	$\Cb(\HK)$\nbd semigroup.
	
	For $\Phit\in\Pit$, $x\in\R_+$, we define
	\begin{equation}\label{eq:gammaPhi}
		\gammaPhi(x) := \sup_{(T, \uvec, \nu)\in\Hn,\,\|\nu\|=x} \bigl|\Phit(T, \uvec,\nu)\bigr|
	\end{equation}
	and note that $\lim_{x\to\infty}\gammaPhi(x) =0$.
	Using Fubini's theorem, we obtain for $\Psit=\Psi^{1, n,\Phit}\in\Ft^1$ and
	$\smallx=(T, \mu, \nu)\in\HK$
	\begin{equation}\label{eq:StPsi}\begin{aligned}
		S_t\Psit(\smallx) &= \intmu{\mathbb P^\smallx\(\pi_t^\smallx(\rspanu)=0\)
			\cdot\tilde\Phi(\tauxn(\uvec)\bigr)} \\
		&=\intmu{e^{-t\nu(\rspanu)} \cdot \Phit\(\tauxn(\uvec)\)}.
	\end{aligned}\end{equation}
	Therefore,
	\begin{equation*}
		\sup_{\smallx\in\HK} \bigl| S_t\Psit(\smallx)-\Psit(\smallx) \bigr|
			\le K^n \sup_{x\in\R_+} \gammaPhi(x)\(1-e^{-tx}\) \convto[0]{t} 0.\qedhere
	\end{equation*}
\end{proof}\sm

\subsection{The infinitesimal generator}
\label{sec:generator}

In this subsection we calculate the action of the generator on the test functions
$\tilde\Psi\in\tilde{\mathcal F}^1$. For these functions to be bounded, we have to work on the space $\HK$.
Note that $\HK$ is a good state space for the pruning process, as once started in $\HK$, it will never leave the space.
In the following we write
\begin{equation}
   \bigl(\Omega,\Dom(\Omega)\bigr) \quad\mbox{ and }\quad \bigl(\Omega_K,\Dom(\Omega_K)\bigr)
\end{equation}
for the infinitesimal generators of the pruning process with state spaces $\Hbi$ and $\HK$ respectively.

\begin{proposition}[Infinitesimal Generator]
	For every\/ $K>0$, we have\/ $\Ft^1\subseteq\Dom(\Omega_K)$. Furthermore, for\/ $\Psit=\Psitpol[1]\in\Ft^1$ and\/
	$\smallx=(T,\mu,\nu)\in\HK$,
	\begin{align}\label{eq:gen}
	    \Omega\Psit\bigl(\smallx\bigr) &=
		    \int\nu(\mathrm{d}v)\bigl[\Psit\bigl((T^{v},\mu,\nu)\bigr)-\Psit\bigl(\smallx\bigr)\bigr] \\
		\label{eq:genpsi}
	     &= -\int\mu^{\otimes n}(\mathrm{d}\underline{u})\;\nu\bigl(\rspanu\bigr)\tilde\Phi\bigl(\tauxn(\uvec)\bigr).
	\end{align}
\label{prop:gen}
\end{proposition}

\begin{proof}
Using Formula~\eqref{eq:StPsi}, we obtain for $\tilde\Psi=\Psitpol[1]\in\tilde{\mathcal F}^1$, $\smallx\in\HK$,
\begin{equation}\label{eq:OmegaPsi}
	\tfrac{1}{t}\(S_t\Psit(\smallx) -\Psit(\smallx)\)
	 = -\tfrac{1}{t} \int\mu^{\otimes n} (\mathrm{d}\underline{u})\,
		\(1-e^{-t\nu(\rspanu)}\)\Phit(\tauxn(\uvec)\bigr).
\end{equation}
Note that $|1-e^{-x}-x|\le x^2$, for all $x\ge 0$ and recall the definition of $\gammaPhi$ from
\eqref{eq:gammaPhi}. Comparing \eqref{eq:OmegaPsi} to \eqref{eq:genpsi}, we see that
\begin{equation}
\begin{aligned} \label{e:027}
   &\sup_{\smallx\in\HK} \Big| \tfrac{1}{t}\(S_t\Psit(\smallx) -\Psit(\smallx)\)
   	+\int\mu^{\otimes n}(\mathrm{d}\underline{u})\,\nu\bigl(\rspanu\bigr)\Phit\(\tauxn(\uvec)\)\Big|
   \\
   &\le \sup_{\smallx\in\HK} t \cdot \int\mu^{\otimes n} (\mathrm{d}\underline{u})\,
	\nu\(\rspanu\)^2\,\bigl|\Phit(\tauxn(\uvec))\bigr| \le tK^n \sup_{x\in\R_+} x^2\gammaPhi(x).
\end{aligned}
\end{equation}
Due to our assumptions on $\Phit\in\Pit$, $x^2\gammaPhi(x)$ is bounded, and we obtain uniform convergence of
$\frac1t(S_t\Psit - \Psit)$ on $\HK$ for $t\to 0$. Hence $\Ft^1\subseteq\Dom(\Omega_K)$ and
Formula~(\ref{eq:genpsi}) are proven.  \sm

We next prove Formula~(\ref{eq:gen}). Notice that for all $\uvec\in T^n$,
\begin{equation}\label{e:029}
	\nu(\rspanu) = \intset{T}{\mathds 1_{\{v\in \rspanu\}}}{\nu}{v}
	= \intset{T}{1-\mathds 1_{\{\uvec\in (T^v)^n\}}}\nu{v}.
\end{equation}
Inserting the latter into \eqref{eq:genpsi} and using Fubini's theorem yields
\begin{equation} \label{e:030}
   \Omega\Psit(\smallx) =
    \int_T\nu(\mathrm{d}v)\,\biggl(\int_{(T^v)^n}\mu^{\otimes n}(\mathrm{d}\underline{u})\,\tilde\Phi\bigl(\tauxn(\uvec)\bigr)-
    \int_{T^n}\mu^{\otimes n}(\mathrm{d}\underline{u})\,\tilde\Phi\bigl(\tauxn(\uvec)\bigr)\biggr),
\end{equation}
which gives (\ref{eq:gen}).
\end{proof}\sm

\comment{
In this subsection we want to give the action of the generator on the test functions $\tilde\Psi\in\tilde{\mathcal F}^1$.
For these to be in the domain of the generator, we have to work on a subspace of bi-measure \Rtree s with
(uniformly) bounded second moments of the total pruning measure of the path to the root from a randomly sampled
point.

\begin{definition}[The space $\HC$]	
	For $C,K\in\R_+$, we define the subspace
	\begin{equation}
		\HC := \Bset{(T,\mu,\nu)\in \HK}{\int\mu^{\otimes n}(\mathrm{d}u)\,\nu^2\bigl([\rho, u]\bigr)\le C}
	\end{equation}
\end{definition}

\begin{remark}\label{rem:Cn}
	Let $C, K\in\R_+$ and $\smallx\in\HC$. Then for all $n\in\N$, we obtain
	\begin{equation}
		\plainint{\nu^2\(\rspan{\cdot}\)}{\mu^{\otimes n}} \le \intmu{\Bigl(\sum_{k=1}^n \nu\([\rho,
		u_k]\)\Bigr)^2} \le C_n < \infty,
	\end{equation}
	for $C_n := nCK^{n-1} + (n^2-n)(C+K)^2K^{n-2}$.
\end{remark}

Note that $\HC$ is a good state spaces for the pruning process, as once started in $\HC$, it will never leave the space.
In the following we write
\begin{equation}
   \bigl(\Omega,\Dom(\Omega)\bigr) \quad\mbox{ and }\quad \bigl(\Omega_C,\Dom(\Omega_C)\bigr)
\end{equation}
for the infinitesimal generators of the pruning process with state spaces $\Hbi$ and $\HCC$ respectively.

\begin{proposition}[Infinitesimal Generator]
	For every\/ $C>0$, we have\/ $\Ft^1\subseteq\Dom(\Omega_C)$. Furthermore, for\/ $\Psit=\Psitpol[1]\in\Ft^1$ and\/
	$\smallx=(T,\mu,\nu)\in\HCC$,
	\begin{align}\label{eq:gen}
	    \Omega\Psit\bigl(\smallx\bigr) &=
		    \int\nu(\mathrm{d}v)\bigl[\Psit\bigl((T^{v},\mu,\nu)\bigr)-\Psit\bigl(\smallx\bigr)\bigr] \\
		\label{eq:genpsi}
	     &= -\int\mu^{\otimes n}(\mathrm{d}\underline{u})\;\nu\bigl(\rspanu\bigr)\tilde\Phi\bigl(\tauxn(\uvec)\bigr).
	\end{align}
\label{prop:gen}
\end{proposition}

\begin{proof}
Using Formula~\eqref{eq:StPsi}, we obtain for $\tilde\Psi=\Psitpol[1]\in\tilde{\mathcal F}^1$, $\smallx\in\HCC$,
\begin{equation}\label{eq:OmegaPsi}
	\tfrac{1}{t}\(S_t\Psit(\smallx) -\Psit(\smallx)\)
	 = -\tfrac{1}{t} \int\mu^{\otimes n} (\mathrm{d}\underline{u})\,
		\(1-e^{-t\nu(\rspanu)}\)\Phit(\tauxn(\uvec)\bigr).
\end{equation}
Note that $|1-e^{-x}-x|\le\tfrac{x^2}{2}$, for all $x\ge 0$, and recall the constants $C_n$ defined in
Remark~\ref{rem:Cn}. Comparing \eqref{eq:OmegaPsi} to \eqref{eq:genpsi}, we see that
\begin{equation}
\begin{aligned} \label{e:027}
   &\sup_{\smallx\in\HCC} \Big| \tfrac{1}{t}\(S_t\Psit(\smallx) -\Psit(\smallx)\)
   	+\int\mu^{\otimes n}(\mathrm{d}\underline{u})\,\nu\bigl(\rspanu\bigr)\Phit\(\tauxn(\uvec)\)\Big|
   \\
   &\le \sup_{\smallx\in\HCC} \tfrac{t}{2} \int\mu^{\otimes n} (\mathrm{d}\underline{u})\,
		\nu\(\rspanu\)^2\,\|\Phit\|_\infty \le \tfrac t2 C_n\|\Phit\|_\infty \convto[0]t 0,
\end{aligned}
\end{equation}
so that  $\tilde{\mathcal F}^1\subseteq\Dom(\Omega_C)$ and Formula~(\ref{eq:genpsi}) are proven.  \sm

We next prove Formula~(\ref{eq:gen}). Notice that for all $\uvec\in T^n$,
\begin{equation}\label{e:029}
	\nu(\rspanu) = \intset{T}{\mathds 1_{\{v\in \rspanu\}}}{\nu}{v}
	= \intset{T}{1-\mathds 1_{\{\uvec\in (T^v)^n\}}}\nu{v}.
\end{equation}
Inserting the latter into \eqref{eq:genpsi} and using Fubini's theorem yields
\begin{equation} \label{e:030}
   \Omega\Psit(\smallx) =
    \int_T\nu(\mathrm{d}v)\,\biggl(\int_{(T^v)^n}\mu^{\otimes n}(\mathrm{d}\underline{u})\,\tilde\Phi\bigl(\tauxn(\uvec)\bigr)-
    \int_{T^n}\mu^{\otimes n}(\mathrm{d}\underline{u})\,\tilde\Phi\bigl(\tauxn(\uvec)\bigr)\biggr),
\end{equation}
which gives (\ref{eq:gen}).
\end{proof}\sm
}


\section{Examples}
\label{sec:examples}
In this section we want to apply Theorem \ref{theo:main} to obtain convergence of various pruning processes that appear in the literature.
We first recall the excursion representation of a measure \Rtree. We denote by
\begin{equation}
\label{e:exc}
   \mathcal E:=\bigl\{e\colon [0,1]\to\R_+ \bigm| e \text{ is l.s.c.},\, e(0)=e(1)=0\bigr\}
\end{equation}
the set of lower semi-continuous excursions on $[0,1]$. From each excursion $e\in \mathcal E$, we can define a
measure \Rtree\ in the following way:

\begin{itemize}
	\item $r_e(x,y):=e(x)+e(y)-2\inf_{[x,y]}e$ is a pseudo-distance on $[0,1]$,
	\item $x,y\in [0,1]$ are said to be equivalent, $x\sim_e y$, if $ r_e(x,y)=0$,
	\item the image of the projection $\pi_e:[0,1]\rightarrow [0,1]/{\sim_e}$ endowed with the push forward of
		$r_e$ (again denoted $r_e$), i.e.\ $T_e:=(T_e,r_e,\rho_e):=\(\pi_e([0,1]),r_e,\pi_e(0)\)$, is a $0$-hyperbolic space (for example, \cite[Lemma~3.1]{EvaWin2006}).
	\item We endow this space with the probability measure $\mu_e:=\pi_e{}_\ast \lambda_{[0,1]}$ which is
		the push forward of the Lebesgue measure on $[0,1]$.
\end{itemize}
We denote by $g:\mathcal E\to \H_\rho$ the resulting ``glue function'',
\begin{equation}
\label{e:glue}
   g(e)
 :=
   \big(T_e,\mu_e\big),
\end{equation}
which sends an excursion to a rooted probability measure $\R$-tree. The map
$g$ is continuous if $\H_\rho$ is endowed with the Gromov-weak topology, and $\mathcal E$ with the uniform topology (see
\cite[Prop.~2.9]{AbrahamDelmasHoscheit:exittimes} for the case of continuous excursions) or, more generally, with the
weaker {\em excursion topology} introduced in \cite{Loehr:Gromovmetric} (see Theorem~4.8 there).

\begin{example}[An approach via excursions]
	Consider a sequence of random excursions $e_N=(e_N(s),s\in [0,1])\in \mathcal E$, $N\in\N$, that converges in
	distribution (with respect to the uniform, respectively the excursion topology) to $e\in \mathcal E$. For each $N\in\mathbb{N}$, we denote by $(X_t^N)_{t\geq0}$ the
	pruning process started in the bi-measure tree $\smallx_{e_N}:=(T_{e_N},\mu_{e_N},\lambda_{T_{e_N}})\in \Hbi$,
	where $\lambda_{T_{e_N}}$ is the length measure on $T_{e_N}$, and similarly for $(X_t)_{t\geq 0}$ and $\smallx_e$.

	Due to continuity of $g$, we have that $g(e_N)$ converges Gromov-weakly in distribution to $g(e)$.
	By Proposition~\ref{prop:lengthconv}, we obtain the {\rm LWV}-convergence in distribution of $\smallx_{e_N}$ to
	$\smallx_e$, and by Theorem~\ref{theo:main}, we get the Skorohod convergence
		\[ (X^N_t)_{t\geq 0} \toSk (X_t)_{t\geq 0} \]
	as $\Hbi$-valued processes with \LWVtopo. Note that this, in particular, implies Skorohod convergence of the pruning
	processes $(T_{e_N}^{\pi_t}, \mu_{e_N})_{t\ge0}$ as measure \Rtree-valued processes in the usual
	Gromov-weak topology, where we do not keep track of the pruning measure.
\label{ex:convexcur}
\end{example}\sm

We shall apply this example to Galton-Watson trees. Consider a critical or sub-critical Galton-Watson tree $\G$ with
offspring distribution $\eta$ on $\N_0$, i.e., every node in the discrete tree has a random number of children given
independently by the distribution $\eta$, where $\E[\eta]\leq 1$. Encode $\G$ as a rooted \Rtree\ with unit
length edges. For each $N\in\mathbb{N}$, let $\G_N$ be the tree $\G$ conditioned to have $N$ nodes (in addition to the root).
We consider two different sampling measures $\mu$ on $\G_N$:
one is the {\em normalized length measure}
\begin{equation}
\label{e:muske}
   \mu_N^{\rm ske}
 :=
   \tfrac1N\lambda_{\G_N},
 \end{equation}
and the second is the {\em uniform measure on the nodes},
\begin{equation}
\label{e:munod}
   \mu_N^{\rm nod}
 :=
    \tfrac 1N\sum_{i=1}^{N} \delta_{x_i},
 \end{equation}
where $\{x_1,..., x_N\}$ are the nodes of $\G_N$. Notice that
\begin{equation}
\label{e:014}
   \mu_N^{\rm nod}\big(A\big)
 =
   \sum_{x\in {\rm nod}(A)} \mu_N^{\rm ske}\big([x_-,x]\big)
 \leq
   \mu_N^{\rm ske}\big(\set{v\in\G_N}{r_{\G_N}(v,A)<1}\big)
\end{equation}
 where ${\rm nod}(A)$ is the set of nodes in $A$ and $x_-$ is the parent of $x$.

In order to obtain convergence, we rescale the tree $\G_N$ to have edge lengths $a_N>0$, i.e., we leave the set
unchanged and multiply the metric by $a_N$. We denote the rescaled tree by $a_N\G_N$. As
\begin{equation}
\label{e:leaN}
   d^{a_N\G_N}_{\rm Pr}\big(\mu_N^{\rm ske}, \mu_N^{\rm nod}\big) \le a_N
\end{equation}
on the rescaled tree by (\ref{e:014}),
$\mu_N^{\rm nod}$ and $\mu_N^{\rm ske}$ become arbitrary close
whenever $a_N$ converges to zero, as $N\to\infty$.

We also consider two different pruning measures $\nu$: one is the \emph{length measure on the rescaled tree},
\begin{equation}
\label{e:015}
   \nu_N^{\rm ske} := \lambda_{a_N\G_N} = a_N\cdot N\cdot\mu_N^{\rm ske},
\end{equation}
and the second is a suitably rescaled \emph{uniform measure on the nodes},
\begin{equation}
   \nu_N^{\rm nod} := a_N\cdot N\cdot\mu_N^{\rm nod}.
\end{equation}

In order to be in a position to apply Example~\ref{ex:convexcur}, we associate the conditioned and rescaled
bi-measure Galton-Watson tree with an excursion.  That is,
by the depth-first search algorithm we obtain a graph-theoretic path $\rho=y_0, y_1,..., y_{2N-1}, y_{2N} =\rho$ in
the discrete tree, which traverses each edge exactly twice. The {\em contour process}
$(C_N(t),0\leq t\leq 1)$ of $\G_N$ is the linear interpolation of $C_N(\frac{k}{2N}) := h(y_k) := r_{\G_N}(\rho, y_k)$,
$k=0,..., 2N$. Note that in our definition of $C_N$, the domain is normalized to $[0,1]$, and we obtain that
\begin{equation}
   g(C_N)=\big(\G_N, \mu_N^{\rm ske}\big).
\end{equation}\sm


\begin{example}[Brownian CRT]\label{ex:CRT}
Let the variance $\sigma^2$ of $\eta$ be finite and choose
\begin{equation}
   a_N
 :=
   \tfrac{\sigma}{\sqrt{N}}.
\end{equation}

We know from Theorem~23 in \cite{Ald1993} that $(a_NC_N(t),0\leq t\leq 1)$ converges uniformly in distribution to
$(2B(t),0\leq t\leq 1)$, where $B$ is the standard Brownian excursion. We now apply Example~\ref{ex:convexcur} and get
the {\rm LWV}\nbd convergence in distribution of the bi-measure \Rtree s
\begin{equation}
\label{e:CRT}
   \big(\tfrac{\sigma}{\sqrt{N}}\G_N, \mu_N^{\rm ske}, \nu_N^{\rm ske}\big)
 \toLWVN
   \big(CRT, \mu, \lambda_{CRT}\big),
\end{equation}
where $(CRT,\mu)=g(2B)$ is the \Rtree\ called Brownian continuum random tree, and $\lambda_{CRT}$ is the length measure
on the Brownian CRT.

By Corollary~\ref{cor:changemu} and Lemma~\ref{lem:changenu}, we also have the convergence
\begin{equation}
\label{e:CRTchoices}{\big(\tfrac{\sigma}{\sqrt{N}}\mathcal G_N, \mu_N, \nu_N\big)\toLWVN \big(CRT, \mu, \lambda_{CRT}\big)}
\end{equation}
for all choices of $\mu_N\in\{\mu_N^{\rm ske},\mu_N^{\rm nod}\}$ and $\nu_N\in\{\nu_N^{\rm ske},\nu_N^{\rm nod}\}$. Finally we have the convergence of the
pruning processes in Skorohod space:
\begin{equation}
\left(\tfrac{\sigma}{\sqrt{N}}\mathcal G_N^{\pi_t},\mu_N, \nu_N\right)_{t\geq0} \toSk[LWV] \left(CRT^{\pi_t},\mu, \lambda_{CRT}\right)_{t\geq0}.
\end{equation}

In particular,
\begin{equation}
\left(\tfrac{\sigma}{\sqrt{N}}\mathcal G_N^{\pi_t},\mu_N^{\rm nod}\right)_{t\geq0} \toSk[Gw] \left(CRT^{\pi_t},\mu\right)_{t\geq0}.
\end{equation}
Notice that for $\nu_N=\nu_N^{\rm ske}$, the pruning process $(\mathcal G_N^{\pi_t})_{t\geq0}$ is, up to the time
transformation $u=e^{-t / \sqrt{N}}$, the same as the pruning process $(\mathcal G^{\rm AP}_u)_{u\in [0,1]}$ uniformly
on the edges of Aldous and Pitman in \cite{AldousPitman1998}. The process on the right hand side is the one considered
by Aldous and Pitman \cite{AldousPitman1998b} and by Abraham and Serlet \cite{AbrahamSerlet2002} for example.
\label{Exp:003}
\end{example}\sm

\begin{example}[$\alpha$-stable L\'{e}vy tree]
We know from Theorem 3.1 of \cite{Duq2003} that if $\eta$ is in the domain of attraction of an $\alpha$-stable
distribution with $\alpha\in (1,2]$, then there exists a sequence $a_N$ such that $(a_N C_N(t),0\leq t\leq 1)$ converges
uniformly in distribution to $(H(t),0\leq t\leq 1)$, where $H$ is a continuous excursion that codes an $\alpha$-stable
L\'{e}vy tree, $(LT_\alpha,\mu):=g(H)$. More precisely, for $\eta(k)\sim_{k\rightarrow \infty}Ck^{-1-\alpha}$, we have $a_N=N^{-\bar\alpha}\left(\frac{\alpha(\alpha-1)}{C\Gamma(2-\alpha)}\right)^{-1/\alpha}$ with $\bar\alpha=1-1/\alpha$ (see Section 1.2 in \cite{CurHaa2012}). As in Example \ref{ex:CRT}, we obtain
\begin{equation}
\left(a_N\mathcal G_N^{\pi_t},\mu_N, \nu_N\right)_{t\geq0} \toSk[LWV] \left(LT_\alpha^{\pi_t},\mu, \lambda_{LT_\alpha}\right)_{t\geq0}
\end{equation}
or more precisely
\begin{equation}
\left(\frac{1}{N^{\bar\alpha}}\mathcal G_N^{\pi_t},\mu_N, \nu_N\right)_{t\geq0} \toSk[LWV] \left(\left(\frac{\alpha(\alpha-1)}{C\Gamma(2-\alpha)}\right)^{1/\alpha}LT_\alpha^{\pi_t},\mu, \lambda_{LT_\alpha}\right)_{t\geq0}.
\end{equation}
where $\mu_N=\mu_N^{\rm ske}$ or $\mu_N^{\rm nod}$ and $\nu_N= \nu_N^{\rm ske}$ or $\nu_N^{\rm nod}$.
\label{Exp:004}
\end{example}\sm

\begin{example}[Pruning at a height]
As before we consider the Gromov-weak convergence $\left(a_N\mathcal G_N,\mu_N\right) \toGw \left(LT_\alpha,\mu\right)$.
For $a\ge 0$, we define the pruning measure
\begin{equation}
\label{e:pruneheight}
   \nu^a_N
 :=
   \sum_{x\in \mathcal G_N^a} \delta_x,
\end{equation}
where
$\mathcal G_N^a=\{x\in \mathcal G_N \mid r_N(\rho,x)=a\}$,
and the corresponding measure
\begin{equation}
	\nu_\infty^a := \sum_{x\in LT_\alpha^a \cap \muskel[LT_\alpha]} \delta_x
\end{equation}
on $LT_\alpha$. Here, we restrict the pruning measure to the points of $LT_\alpha$ which are not leaves in order
to ensure the condition $\nu_\infty^a\(\muleaf[LT_\alpha^a]\)=0$.
Because the probability that $\mu(LT_\alpha^a)\ne 0$ is zero for fixed $a$,
the sequence $(\nu_N^a)_{N\in \N\cup\{\infty\}}$ almost surely depends \bnuas[(LT_\alpha, \mu)]\ continuously on the distances, i.e.
	\[R^{LT_\alpha}{}_\ast \mu^{\otimes n} \left( {\rm Discont }(F_n) \right)=0 \;\; a.s.,\]
for $F_n$ as in \eqref{eq:lengthdepend}.  We use Proposition \ref{prop:lengthconv} and the previous construction to get
\begin{equation}
   \big(a_N\mathcal G_N^{\pi_t},\mu_N, \nu_N^a\big)_{t\geq0} \toSk \big(LT_\alpha^{\pi_t},\mu, \nu_\infty^a\big)_{t\geq0}.
\end{equation}
It is easy to check that $\left(LT_\alpha^{\pi_t},\mu, \nu_\infty^a\right)$ converges almost surely, as $t\rightarrow\infty$, in the {\rm LWV}-topology to $\left( LT_\alpha^{\leq a},\mu, 0 \right)$ where $LT_\alpha^{\leq a}=\{x\in LT_\alpha \mid r(\rho,x)\leq a\}$. This is the pruning construction at the height $a$ of Miermont \cite{Miermont2003}.
\label{Exp:005}
\end{example}\sm

\begin{remark}[Pruning based on other scaling results]
Some authors give other convergence of Galton-Watson trees to continuous trees. For example a sequence of
Galton-Watson trees $(\mathcal G_N)_{N\in \N}$ conditioned to have maximum height at least $\gamma_N T$
converges to a general L\'{e}vy tree conditioned to have maximum height at least $T$, see Proposition 2.5.2 in
\cite{DuqGall2002}. Or a sequence of Galton-Watson trees that converges to a forest of L\'{e}vy trees, see
Theorem 2.4.1 in \cite{DuqGall2002}. In the first case, the previous results clearly apply. In the second
case, in general we do not have an excursion with finite length anymore, i.e., the measure $\mu^{\rm ske}$ might become infinite. However, if we restrict the domain of the contour processes to a finite interval, we can still apply the previous results.
\label{Rem:005}
\end{remark}\sm

\begin{example}[More general pruning]
A non-uniform pruning process on the branch points of a general Galton-Watson tree has been defined by Abraham, Delmas
and He \cite{AbrDelHe2012}: they cut a branch point $v$ and its subtree above independently with probability $1-u^{c(v)-1}$,
where $c(v)$ is the number of children of $v$. This corresponds to taking the pruning measure $\nu^{ADH}_N$ on
$\mathcal G_N$ that is supported on the branch points and satisfies
\begin{equation}
\label{e:nuADH}
   \nu^{ADH}_N\(\{v\}\)
 :=
   c(v)-1.
\end{equation}

A pruning process on the infinite branch points of a L\'{e}vy tree has been defined by Abraham and Delmas
\cite{AbrahamDelmas2012}: they cut each infinite branch point and its subtree above independently with probability
$1-e^{-t\Delta_x}$ where $\Delta_x$ is the weight of the node $x$ that can be defined using the jumps of the L\'{e}vy
process. This corresponds to taking a measure $\nu^{AD}$ on the infinite branch points of the L\'{e}vy tree.

Because we know that a properly renormalized sequence of conditioned Galton-Watson trees converges to a L\'{e}vy tree, we conjecture that there exists a sequence $b_N$ such that
\begin{equation}
  \big( a_N\mathcal G_N , \mu^{\rm nod}_N , \nu_N^{\rm ske}+b_N \nu_N^{\rm ADH}\big) \toLWVN \big( LT , \mu^{\rm nod} , \nu^{\rm ske}+ \nu^{\rm AD}\big)
\end{equation}
where LT is a L\'{e}vy tree or at least an $\alpha$-stable L\'{e}vy tree with $b_N$ of the order $N^{-1/\alpha}$ up to a
slowly varying function. The Poisson point process with intensity $\nu^{\rm ske}+ \nu^{\rm AD}$ used in the pruning of the L\'{e}vy tree is the Poisson point process
given in Subsection~4.2 of~\cite{Voisin2011}.
\label{Exp:001}
\end{example}\sm

\begin{ex}[Cutting down trees]
Random deconstruction of trees is an old topic which has recently gained a lot of attention (compare,
\cite{MeirMoon1970,Panholzer2006,Janson2006,DrmotaIksanovMohleRosler2009,Holmgren2010,Bertoin2010,BertoinMiermont2012}). The main result of \cite{Janson2006}
is the following. Given a finite-variance Galton-Watson tree conditioned to have $N$ nodes, select an edge at random and
delete the subtree above. Repeat the procedure until the root is isolated.
Then the suitably rescaled number of cuts needed converges jointly with the rescaled tree to some random couple
$(Z_T,T)$. It is known that the limiting tree $T$ is the Brownian CRT, while (unconditioned) $Z_T$ is Rayleigh
distributed. In a very recent paper, Abraham and Delmas \cite{AbrahamDelmas2013} used a pruning with the length
measure on the Brownian CRT (compare Example \ref{ex:CRT}) and showed that given $T$, $Z_T$ equals in
distribution the averaged time it takes to separate a point from the root. The latter quantity was used in the
proof given by Janson \cite{Janson2006}.  In this example, we show that whenever bi-measure $\R$-trees converge
-- provided some extra tightness conditions hold -- Janson's quantities converge as well.

Let $({\mathcal G}_N,\mu_N,\nu_N)_{N\in\mathbb{N}\cup\{\infty\}}$ be a sequence of random bi-measure $\R$-trees
such that
\begin{equation}
\label{e:024}
   \big({\mathcal G}_N,\mu_N,\nu_N\big)\ToLWVN \big({\mathcal G}_\infty,\mu_\infty,\nu_\infty\big).
\end{equation}
For each $N\in\mathbb{N}\cup\{\infty\}$, let the pruning process $(X^N_t)_{t\ge 0}$ start in $X_0^N=({\mathcal G}_N,\mu_N,\nu_N)$.
Denote by $\Theta_N$ the {\em averaged time until a point gets separated from the root $\rho_N$},
where the average is taken with respect to the sampling measure $\mu_N$.
Given a realization $\smallx\in \Hbi$ of $X^N_0$, consider for each $u\in\supp(\mu_N)$ the (random) time $\mathcal E_\smallx^u$ until
$u$ gets separated from $\rho_N$, i.e., until a cut point falls on $[\rho_N,u[$. We abbreviate $\mathcal E_N^u:= \mathcal E_{X_0^N}^u$
and obtain
\begin{equation}\label{e:Theta}
	\Theta_N = \int_{{\mathcal G}_N}\mu_N(\mathrm{d}u)\,{\mathcal E}_N^u.
\end{equation}
For all finite subsets $\{u_1,...,u_n\}\subseteq{\mathcal G}_N$ and $t_1,...,t_n\ge 0$, the distribution of
$\CE_N^{u_1},\ldots,\CE_N^{u_n}$ is given by
\begin{equation} \label{e:density}
	\mathbb{P}\big({\mathcal E}_N^{u_1}\ge t_1,...,{\mathcal E}_N^{u_n}\ge t_n\bigm|({\mathcal G}_N,\mu_N,\nu_N)\big)
	  = \prod_{l=1}^n  e^{-t_{p(l)} \cdot \nu_N(S_l \setminus S_{l+1})},
\end{equation}
where $p\colon \{1,2,...,n\}\to\{1,2,...,n\}$ is any permutation such that $t_{p(1)} \le \cdots \le t_{p(n)}$,
and $S_l:= \rspan{u_{p(l)},...,u_{p(n)}}$.
Then for all $n\in\mathbb{N}$,
\begin{equation}\label{e:Thetan}
\begin{aligned}
	\mathbb{E}\bigl[\Theta_N^n\bigr] &= \mathbb{E}\Bigl[\int_{{\mathcal G}^n_N}\mu^{\otimes n}_N(\mathrm{d}(u_1,...,u_n))\,
		\mathbb{E}\big[\prod_{l=1}^n{\mathcal E}_N^{u_l}\bigm|({\mathcal G}_N,\mu_N,\nu_N)\big]\Bigr] \\
	&= n!\cdot\mathbb{E}\Bigl[ \int_{{\mathcal G}^n_N}\mu^{\otimes n}_N(\mathrm{d}(u_1,...,u_n))\,
		\prod_{j=1}^n \frac{1}{\nu_N\big(\rspan{u_1,...,u_j}\big)}\Bigr],
\end{aligned}
\end{equation}
where the last equality is obtained by using (\ref{e:density}) and easy computations with the formula
${\mathbb E\bigl[\prod_{i=1}^n X_i\bigr]=\int_{\R_+^n} \mathbb P(X_i>t_i, \forall i)\,{\rm d}(t_1 \dots t_n)}$.

\noindent Now assume the following:
\begin{enumerate}
\item For all $n\in\mathbb{N}$, $\eps>0$ there is an $M>0$ such that
	\begin{equation}\label{e:037}
		\sup_{N\in\N} \E\biggl[\int_{\G^n_N}\mu^{\otimes n}_N(\mathrm{d}(u_1,...,u_n))\,\Bigl(\prod_{j=1}^n
		\frac{1}{\nu_N\bigl(\rspan{u_1,...,u_j}\bigr)} - M\Bigr)^+\biggr] \le \eps.
	\end{equation}
\item\label{it:momuniq}
	There is only one probability measure $\mathbb{Q}$ on $\R_+$ with moments
	\begin{equation}\label{e:038}
	   \int_{\R_+}\mathbb{Q}(\mathrm{d}\theta)\,\theta^n
		=n!\cdot \mathbb{E}\Bigl[\int\mu^{\otimes n}_\infty(\mathrm{d}(u_1,...,u_n))\prod_{j=1}^n \frac{1}{\nu_\infty(\rspan{u_1,...,u_j})}\Bigr]
	\end{equation}
	for each $n\in\mathbb{N}$.
\end{enumerate}
Note that these assumptions are in particular satisfied in the case of conditioned finite variance Galton-Watson
trees converging to the Brownian CRT if $\nu_N$ is the length measure and $\mu_N$ the uniform distribution on
the nodes (see, e.g., \cite[proof of Lem.~4.5, Thm.~1.9]{Janson2006}).

For each $n,M\in\mathbb{N}$, define $\gamma_M^n\colon \Hn \to\R_+$ by
\begin{equation}
	\gamma_M^n\(T,(u_1,\ldots,u_n), \nu\) := M \land \prod_{j=1}^n \nu\(\rspan{u_1,\ldots,u_j}\)^{-1}.
\end{equation}
Then $\gamma^n_M\in{\mathcal C}_b(\H_n)$ if $\H_n$ is equipped with the \sGwtopo, and the \LWVconv\
(\ref{e:024}) together with Proposition~\ref{prop:equivconvergence} implies that
\begin{equation}
	\E\Bigl[\plainintmu[N]{\gamma^n_M\circ \tau_{({\mathcal G}_N,\mu_N,\nu_N)}}\Bigr] \tNo
	\E\Bigl[\plainintmu[\infty]{\gamma^n_M\circ \tau_{({\mathcal G}_\infty,\mu_\infty,\nu_\infty)}}\Bigr].
\end{equation}
Thus, we also have
$\mathbb{E}\big[\Theta_N^n\big]\tNo\mathbb{E}\big[\Theta_\infty^n\big]$ for each $n\in\mathbb{N}$, provided that
(\ref{e:037}) holds. By assumption~\ref{it:momuniq}, the moments of $\Theta_\infty$ determine its distribution
uniquely, and therefore the method of moments yields
\par\vspace*{\abovedisplayskip}\hfill $\displaystyle\Theta_N\TNo\Theta_\infty.$ \exend
\end{ex}\sm

\noindent{\sc Acknowledgement. }
The authors would like to thank the anonymous referees for several helpful remarks and references.

\ifonlythms\end{xcomment}\fi
\bibliography{LWV}

\begin{thebibliography}{DIMR09}

\bibitem[AD08]{AbrDel2008}
Romain Abraham and Jean-Fran{\c c}ois Delmas.
\newblock {Fragmentation associated with {L\'evy} processes using snake}.
\newblock {\em Probability Theory and Related Fields}, 141:113--154, 2008.

\bibitem[AD12]{AbrahamDelmas2012}
Romain Abraham and Jean-{Fran\c{c}ois} Delmas.
\newblock A continuum-tree-valued {Markov} process.
\newblock {\em Ann. of Probab.}, 40(3):1167--1211, 2012.

\bibitem[AD13]{AbrahamDelmas2013}
Romain Abraham and Jean-Fran{\c{c}}ois Delmas.
\newblock Record process on the continuum random tree.
\newblock {\em ALEA Lat. Am. J. Probab. Math. Stat.}, 10(1):225--251, 2013.

\bibitem[ADH12]{AbrDelHe2012}
Romain Abraham, Jean-{Fran\c{c}ois} Delmas, and Hui He.
\newblock Pruning {Galton-Watson} trees and tree-valued {Markov} process.
\newblock {\em Ann. Inst. H. {Poincar\'e} Probab. Statist.}, 48(3), 2012.

\bibitem[ADH13]{AbrDelHos2013}
Romain Abraham, Jean-Fran\c{c}ois Delmas, and Patrick Hoscheit.
\newblock A note on the {G}romov-{H}ausdorff-{P}rokhorov distance between
  (locally) compact metric measure spaces.
\newblock {\em Electron. J. Probab.}, 18(14):1--21, 2013.

\bibitem[ADH14]{AbrahamDelmasHoscheit:exittimes}
Romain Abraham, Jean-Fran{\c{c}}ois Delmas, and Patrick Hoscheit.
\newblock Exit times for an increasing {L}\'evy tree-valued process.
\newblock {\em Probab. Theory Related Fields}, 159(1-2):357--403, 2014.

\bibitem[ADV10]{AbrDelVoi2010}
Romain Abraham, Jean-Fran{\c{c}}ois Delmas, and Guillaume Voisin.
\newblock Pruning a {L\'evy} continuum random tree.
\newblock {\em Electron. J. Probab.}, 15(46):1429--1473, 2010.

\bibitem[Ald91]{Ald1991a}
David Aldous.
\newblock The continuum random tree {I}.
\newblock {\em Ann. Probab.}, 19:1--28, 1991.

\bibitem[Ald93]{Ald1993}
David Aldous.
\newblock The continuum random tree {III}.
\newblock {\em Ann. Probab.}, 21:248--289, 1993.

\bibitem[ALW14]{AthreyaLoehrWinter}
Siva Athreya, Wolfgang L{\"o}hr, and Anita Winter.
\newblock The gap between {Gromov}-vague and {Gromov-Hausdorff}-vague topology.
\newblock arXiv:1407.6309, 2014.

\bibitem[AP98a]{AldousPitman1998b}
David Aldous and Jim Pitman.
\newblock The standard additive coalescent.
\newblock {\em Ann. Probab}, 26(4):1703--1726, 1998.

\bibitem[AP98b]{AldousPitman1998}
David Aldous and Jim Pitman.
\newblock Tree-valued {Markov} chains derived from {Galton-Watson} processes.
\newblock {\em Ann. Inst. H. Poincar{\'e} Probab. Statist.}, 34(5):637--686,
  1998.

\bibitem[AS02]{AbrahamSerlet2002}
Romain Abraham and Laurent Serlet.
\newblock Poisson snake and fragmentation.
\newblock {\em Elect. Journal of Probab.}, 7:1--15, 2002.

\bibitem[BBI01]{BurBurIva01}
Dmitri Burago, Yuri Burago, and Sergei Ivanov.
\newblock A course in metric geometry, graduate studies in mathematics.
\newblock {\em AMS, Boston, MA}, 33, 2001.

\bibitem[Ber12]{Bertoin2010}
Jean Bertoin.
\newblock Fires on trees.
\newblock {\em Ann. Inst. Henri Poincar\'e Probab. Stat.}, 48(4):909--921,
  2012.

\bibitem[Bil99]{Billingsley1999}
P.~Billingsley.
\newblock {\em Convergence of {P}robability {M}easures}.
\newblock Wiley, New York, 1999.

\bibitem[BK10]{BloKour2010}
Douglas Blount and Michael~A. Kouritzin.
\newblock On convergence determining and separating classes of functions.
\newblock {\em Stochastic Processes and their Applications},
  120(10):1898--1907, 2010.

\bibitem[BM13]{BertoinMiermont2012}
Jean Bertoin and Gr{\'e}gory Miermont.
\newblock The cut-tree of large {G}alton-{W}atson trees and the {B}rownian
  {CRT}.
\newblock {\em Ann. Appl. Probab.}, 23(4):1469--1493, 2013.

\bibitem[Bog07]{BogachevI2007}
V.~I. Bogachev.
\newblock {\em Measure Theory, Volume {I}}.
\newblock Springer, 2007.

\bibitem[CH12]{CurHaa2012}
Nicolas Curien and Bénédicte Haas.
\newblock The stable trees are nested.
\newblock {\em Probability Theory and Related Fields}, pages 1--37, 2012.

\bibitem[DGP11]{DGP:mmmspace}
Andrej Depperschmidt, Andreas Greven, and Peter Pfaffelhuber.
\newblock Marked metric measure spaces.
\newblock {\em Electron. Commun. Prob.}, 16:174--188, 2011.

\bibitem[DIMR09]{DrmotaIksanovMohleRosler2009}
M.~Dromota, A.~Iksanov, Martin Mohle, and Uwe Rosler.
\newblock A limiting distribution for the number of cuts needed to isolate the
  root of a random recursive tree.
\newblock {\em Random Structures and Algorithms}, 34:319--336, 2009.

\bibitem[DLG02]{DuqGall2002}
Thomas Duquesne and Jean-Fran{\c{c}}ois Le~Gall.
\newblock Random trees, {L\'e}vy processes and spatial branching processes.
\newblock {\em Ast\'erisque}, 281:vi+147, 2002.

\bibitem[DMT96]{DreMouTer96}
Andreas~W.M. Dress, V.~Moulton, and W.F. Terhalle.
\newblock T-theory: An overview.
\newblock {\em Europ. J. Combinatorics}, 17(2-3), 1996.

\bibitem[DT96]{DreTer96}
Andreas~W.M. Dress and W.F. Terhalle.
\newblock The real tree.
\newblock {\em Adv. Math.}, 120:283--301, 1996.

\bibitem[Dud02]{Dudley2002}
M.~Dudley, Richard.
\newblock {\em Real Analysis and probability}.
\newblock Cambridge University Press, Cambridge, 2002.

\bibitem[Duq03]{Duq2003}
Thomas Duquesne.
\newblock A limit theorem for the contour process of conditioned
  {Galton-Watson} trees.
\newblock {\em The Annals of Probab.}, 31(2):996--1027, 2003.

\bibitem[EK86]{EthierKurtz86}
S.N. Ethier and T.~Kurtz.
\newblock {\em {Markov} {P}rocesses. {C}haracterization and {C}onvergence}.
\newblock John Wiley, New York, 1986.

\bibitem[EPW06]{EvaPitWin2006}
Steven~N. Evans, Jim Pitman, and Anita Winter.
\newblock Rayleigh processes, real trees, and root growth with re-grafting.
\newblock {\em Prob. Theo. Rel. Fields}, 134(1):81--126, 2006.

\bibitem[Eva07]{EV2007}
Steven~N. Evans.
\newblock Probability and real trees.
\newblock {\em \'Ecole d'\'Et\'e de Probabilit\'es de Saint Flour XXXV-2005,
  Lecture Notes in Mathematics}, 1920:1--193, 2007.

\bibitem[EW06]{EvaWin2006}
Steven~N. Evans and Anita Winter.
\newblock Subtree prune and re-graft: A reversible real-tree valued {M}arkov
  chain.
\newblock {\em Ann. Prob.}, 34(3):918--961, 2006.

\bibitem[Fuk87]{Fukaya1987}
Kenji Fukaya.
\newblock Collapsing of {Riemannian} manifolds and eigenvalues of {Laplace}
  operator.
\newblock {\em Inventiones Math.}, 87(3):517--547, 1987.

\bibitem[GPW09]{GrePfaWin2009}
Andreas Greven, Peter Pfaffelhuber, and Anita Winter.
\newblock Convergence in distribution of random metric measure spaces (the
  {$\Lambda$}-coalescent measure tree).
\newblock {\em Probab. Theory Related Fields}, 145(1-2):285--322, 2009.

\bibitem[Gro99]{Gromov2000}
Misha Gromov.
\newblock {\em Metric structures for {R}iemannian and non-{R}iemannian spaces},
  volume 152 of {\em Progress in Mathematics}.
\newblock Birkh\"auser Boston Inc., Boston, MA, 1999.

\bibitem[HJ77]{HoffmannJorgensen1977}
J.~Hoffmann-J{\o}rgensen.
\newblock Probability in {Banach} spaces.
\newblock In {\em {{\'Ecole d'\'Et\'e} de {Probabilit\'es} de Saint Flour
  VI-1976}}, volume 598 of {\em Lecture Notes in Mathematics}. Springer, 1977.

\bibitem[Hol10]{Holmgren2010}
C.~Holmgren.
\newblock Random records and cuttings in binary search trees.
\newblock {\em Combinatorics, Probability and Computing}, 19:391--424, 2010.

\bibitem[Jan06]{Janson2006}
S.~Janson.
\newblock Random cuttings and records in deterministic and random trees.
\newblock {\em Random Structures Algorithms}, 29:139--179, 2006.

\bibitem[LC57]{LeCam57}
Lucien Le~Cam.
\newblock Convergence in distribution of stochastic processes.
\newblock {\em Univ. California Publ. Statist.}, 2:207--236, 1957.

\bibitem[L{\"o}h13]{Loehr:Gromovmetric}
Wolfgang L{\"o}hr.
\newblock Equivalence of {G}romov-{P}rohorov- and {G}romov's
  {$\underline\square_\lambda$}-metric on the space of metric measure spaces.
\newblock {\em Electron. Commun. Probab.}, 18:no. 17, 10, 2013.

\bibitem[Lyo92]{Lyons1992}
Russel Lyons.
\newblock Random walks, capacities, and percolation on trees.
\newblock {\em Annals of Probability}, 20:2043--2088, 1992.

\bibitem[Mie03]{Miermont2003}
Gr{\'e}gory Miermont.
\newblock {Self-similar fragmentations derived from the stable tree {I}:
  splitting at heights}.
\newblock {\em Probability Theory and Related Fields}, 127:423--454, 2003.
\newblock 30 pages.

\bibitem[Mie05]{Miermont2005}
Gr\'{e}gory Miermont.
\newblock Self-similar fragmentations derived from the stable tree {II}:
  splitting at nodes.
\newblock {\em Probability Theory and Related Fields}, 131:341--375, 2005.

\bibitem[Mie09]{Miermont2009}
Gr{\'e}gory Miermont.
\newblock Tessellations of random maps of arbitrary genus.
\newblock {\em Ann. Sci. Ec. Norm. Sup}, 42:725--781, 2009.

\bibitem[MM70]{MeirMoon1970}
A.~Meir and J.W. Moon.
\newblock Cutting down random trees.
\newblock {\em J. Australian Math. Soc.}, 11, 1970.

\bibitem[Pan06]{Panholzer2006}
A.~Panholzer.
\newblock Cutting down very simple trees.
\newblock {\em Quaestiones Mathematicae}, 29:211--227, 2006.

\bibitem[Vil09]{Villani2009}
C{\'{e}}dric Villani.
\newblock {\em Optimal Transport}, volume 338 of {\em Grundlehren der
  mathematischen Wissenschaften}.
\newblock Springer, Berlin-Heidelberg, 2009.

\bibitem[Voi11]{Voisin2011}
Guillaume Voisin.
\newblock Dislocation measure of the fragmentation of a general {L\'e}vy tree.
\newblock {\em ESAIM Probab. Stat.}, 15:372--389, 2011.

\end{thebibliography}
\bibliographystyle{alpha}
\end{document}